\documentclass[a4paper,10pt]{article}

\usepackage[dvipdfm,%
 bookmarks=true,%
 bookmarksnumbered=true,%
 colorlinks=true,%
 pdftitle={},%
 pdfauthor={Yokoyama},%
 pdfsubject={},%
 pdfkeywords={}]{hyperref}

\makeatletter
\def\blfootnote{\gdef\@thefnmark{}\@footnotetext}
\makeatother

\usepackage{soul}

\usepackage{latexsym}
\usepackage{amsmath}
\usepackage{amscd}
\usepackage{amssymb}
\usepackage{authblk}
\DeclareSymbolFont{bbold}{U}{bbold}{m}{n}
\DeclareMathSymbol{\bbomega}{\mathord}{bbold}{"7F}

\usepackage{url}
\usepackage{amsthm}

\def\PRA{\mathsf{PRA}}
\def\RCAo{\mathsf{RCA_0}}

\def\WKLo{\mathsf{WKL_0}}
\def\ACAo{\mathsf{ACA_0}}

\def\DNC{\mathrm{DNC}}
\def\DNCZP{\mathrm{2\mbox{-}DNC}}
\def\RRT{\mathrm{RRT}}
\def\RWKL{\mathrm{RWKL}}
\def\CSig{\mathrm{C}\Sigma}

\def\RCA{\mathsf{RCA_0}}

\def\WKL{\mathsf{WKL_0}}
\def\ACA{\mathsf{ACA_0}}

\def\PA{\mathrm{PA}}

\def\RCAs{\mathsf{RCA_0^{*}}}

\def\WKLs{\mathsf{WKL_0^{*}}}

\def\ATRo{\mathsf{ATR_0}}
\def\PCAo{\Pi^1_1\text{-}\mathsf{CA_0}}
\def\wkl{\mathrm{WKL}}

\def\E{\exists}
\def\A{\forall}
\def\N{\mathbb{N}}

\def\Cod{\mathrm{Cod}}

\def\Th{\mathrm{Th}}

\def\card{\mathop{\mathrm{card}}\nolimits}

\def\rest{{\upharpoonright}}

\def\P2{\Pi^1_2}

\newcommand{\Psf}{\mathsf{P}}
\newcommand{\Qsf}{\mathsf{Q}}

\renewcommand{\labelenumi}{$\arabic{enumi}.$}

\def\PHt{\mathrm{PH}^2_2}
\def\RT{\mathrm{RT}}
\def\psRT{\mathrm{psRT}}

\def\SRT{\mathrm{SRT}}

\def\ADS{\mathrm{ADS}}
\def\CAC{\mathrm{CAC}}
\def\SADS{\mathrm{SADS}}

\def\COH{\mathrm{COH}}
\newcommand\EM{\mathrm{EM}}
\newcommand\SEM{\mathrm{SEM}}

\def\Ii{\mathrm{I}\Sigma_1}

\def\BII{\mathrm{B}\Sigma^0_2}
\def\II{\mathrm{I}\Sigma^0_1}
\def\III{\mathrm{I}\Sigma^0_2}

\newcommand{\IN}[1][n]{\mathrm{I}\Sigma^0_{#1}}

\newcommand{\BN}[1][n]{\mathrm{B}\Sigma^0_{#1}}

\renewcommand{\labelenumi}{$\arabic{enumi}.$}

\newcounter{menum}
{\begin{enumerate}%
\setcounter{enumi}{#1}}%
{\setcounter{menum}{\value{enumi}}\end{enumerate}}

\newtheorem{thm}{Theorem}[section]
\newtheorem{theorem}[thm]{Theorem}

\newtheorem*{theorem*}{Theorem}

\newtheorem*{claim*}{Claim}
\newtheorem{prop}[thm]{Proposition}
\newtheorem{lem}[thm]{Lemma}
\newtheorem{cor}[thm]{Corollary}

\newtheorem{lemma}[thm]{Lemma}
\newtheorem{corollary}[thm]{Corollary}

\theoremstyle{definition}
\newtheorem{defi}{Definition}[section]
\newtheorem{definition}[defi]{Definition}
\newtheorem{rem}[thm]{Remark}

\newtheorem{exa}[defi]{Example}

\newtheorem{question}[defi]{Question}

\usepackage{xcolor}

\definecolor{lightred}{rgb}{1,.60,.60}

\begin{document}
\newpage
%\thispagestyle{empty}
%\pagenumbering{roman}
%\setcounter{page}{0}
%\keywords{}

\title{The proof-theoretic strength of Ramsey's theorem for pairs and two colors}

\author[1]{Ludovic Patey}
\author[2]{Keita Yokoyama}

\affil[1]{\small \sf{ludovic.patey@computability.fr}}
\affil[2]{\small \sf{y-keita@jaist.ac.jp}}

% First draft: October 12, 2015
\date{March 17, 2016 (first version)\\ March 15, 2018 (revised version)}

\blfootnote{%
The authors are grateful to Theodore A. Slaman and Emanuele Frittaion for useful comments and discussions.
We also thank Leszek Ko{\l}odziejczyk for pointing out a flaw in Definitions~\ref{def-Gamma-large} and \ref{def-density-with-indicator}.

Ludovic Patey is funded by the John Templeton Foundation (`Structure and Randomness in the Theory of Computation' project).
The opinions expressed in this publication are those of the author(s) and do not necessarily reflect the views of the
John Templeton Foundation.

Keita Yokoyama is partially supported by
JSPS KAKENHI (grant numbers 16K17640 and 15H03634)
 and JSPS Core-to-Core Program (A.~Advanced Research Networks).

Part of this work in this paper was done during the Dagstuhl Seminar 15392 ``Measuring the Complexity of Computational Content: Weihrauch Reducibility and Reverse Analysis''.%
}

\maketitle
\def\Bexp{\mathrm{B}\Sigma_{1}+\mathrm{exp}}
\newcommand\RF{\mathrm{RF}}
\newcommand\tpl{\mathrm{tpl}}
\newcommand\col{\mathrm{col}}
\newcommand\fin{\mathrm{fin}}
\newcommand\Log{\mathrm{Log}}
\newcommand\Ct{\mathrm{Const}}
\newcommand\It{\mathrm{It}}
\renewcommand\PHt{\widetilde{\mathrm{PH}}{}}
\newcommand\BME{\mathrm{BME}_{*}}
\newcommand\HT{\mathrm{HT}}
\newcommand\Fin{\mathrm{Fin}}
\newcommand\FinHT{\mathrm{FinHT}}
\newcommand\wFinHT{\mathrm{wFinHT}}
\newcommand\FS{\mathrm{FS}}
\newcommand\LL{\mathsf{L}}
\newcommand\GPg{\mathrm{GP}}
\newcommand\GP{\mathrm{GP}^{2}_{2}}
\newcommand\FGPg{\mathrm{FGP}}
\newcommand\FGP{\mathrm{FGP}^{2}_{2}}
\newcommand\SGP{\mathrm{SGP}^{2}_{2}}
\newcommand\Con{\mathrm{Con}}
\newcommand\WF{\mathrm{WF}}

\begin{abstract}
Ramsey's theorem for $n$-tuples and $k$-colors ($\RT^n_k$) asserts
that every $k$-coloring of~$[\N]^n$ admits an infinite monochromatic subset.
We study the proof-theoretic strength of Ramsey's theorem for pairs and two colors, namely, the set of its
$\Pi^0_1$ consequences, and show that $\RT^2_2$ is $\Pi^0_3$ conservative over~$\II$.
This strengthens the proof of Chong, Slaman and Yang that $\RT^2_2$ does not imply $\III$,
and shows that $\RT^2_2$ is finitistically reducible, in the sense of Simpson's partial realization
of Hilbert's Program. Moreover, we develop general tools to simplify the proofs of $\Pi^0_3$-conservation theorems.

\medskip



\end{abstract}

\section{Introduction}

Ramsey's theorem for $n$-tuples and $k$-colors ($\RT^n_k$) asserts
that every $k$-coloring of~$[\N]^n$ admits an infinite monochromatic subset.
Ramsey's theorem is probably the most famous theorem of Ramsey's theory,
and plays a central role in combinatorics and graph theory
 (see, e.g., \cite{Comb-and-Graph, book-Ramsey-theory})
with numerous applications in mathematics and computer science,
 among which
functional analysis~\cite{book-Ramsey-in-analysis}
automata theory~\cite{book-automata},
%graph theory~\ludovic{REFS}
 or termination analysis~\cite{PR2004}.
% ,SY201X}.
An important aspect of Ramsey's theorem is its definable class of fast-growing functions.
%One important aspect of Ramsey's theorem is a class of fast-growing functions defined from Ramsey's theorem.
Erd\"os \cite{Erdos1947} showed that the (diagonal) Ramsey number has an exponential growth rate.
Actually, Ramsey's theorem defines much faster-growing functions, which is studied by Ketonen and Solovay \cite{KS81}, among others.
The growth rate of these functions have important applications, since it provides upper bounds to combinatorial questions from various fields.
This type of question is heavily related to proof theory, and with their language, the question is formalized as follows:
\begin{quote}
 \it What is the class of functions whose existence is provable (with an appropriate base system) from Ramsey's theorem?
\end{quote}
For example, the Ramsey number function belongs to this class since the existence of the Ramsey number $R(n, k)$ is guaranteed
by Ramsey's theorem.
%For example, the existence of Ramsey number $R(n,k)$ is guaranteed by Ramsey's theorem, so the Ramsey number function is in this class.
In fact, this class of functions decides the so-called ``proof-theoretic strength'' of Ramsey's theorem.

Ramsey's theorem also plays a very important role in reverse mathematics
as it is one of the main examples of theorems escaping the Big Five phenomenon (see Section~\ref{subsect:ramsey-consequences}).
Reverse mathematics is a general program that classifies theorems by two different measures, namely,
by their computability-theoretic strength and by their proof-theoretic strength.
As it happens, consequences of Ramsey's theorem are notoriously hard to study in reverse mathematics,
and therefore received a lot of attention from the reverse mathematics community.
Especially, determining the strength of Ramsey's theorem for pairs ($\RT^{2}_{2}$) is always a central topic in the study of reverse mathematics.
This study yielded series of seminal papers~\cite{Jockusch1972Ramseys,Seetapun1995strength,CJS,CSY2014} 
introducing both new computability-theoretic and proof-theoretic techniques.
(See Section~\ref{subsect:ramsey-consequences} for more details of its computability-theoretic strength.)

%Hirst~\cite{Hirst-PhD} showed that $\RT^{2}_{2}$ implies the $\Sigma^0_2$-bounding principle ($\BII$).
%On the other hand, Cholak, Jockusch and Slaman~\cite{CJS} showed that $\WKLo+\RT^{2}_{2}+\III$ 
%is a $\Pi^{1}_{1}$-conservative extension of $\III$, where $\WKLo$ stands for weak K\"onig's lemma
%and~$\IN$ is the $\Sigma^0_n$-induction scheme.
%Thus, the first-order strength of Ramsey's theorem for pairs and two colors
% is in between $\BII$ and $\III$.
%After this work, the project of deciding the first-order strength of $\RT^{2}_{2}$ 
%has been strongly carried out using forcing constructions or priority arguments 
%on nonstandard models of $\BII$, mainly by Chong, Slaman and Yang~\cite{CSY2012,CSY2014}.
%They proved in particular that $\WKLo+\RT^{2}_{2}$ does not imply $\III$~\cite{CSY201X}.

In this paper, we mainly focus on the proof-theoretic strength of Ramsey's theorem for pairs.
By the proof-theoretic strength 
of a theory $T$ we mean the set of $\Pi^0_1$ sentences which are provable in $T$,
% or the proof-theoretic ordinal of $T$ which is decided by the $\Pi^{0}_{2}$-consequences of $T$.
 or the proof-theoretic ordinal of $T$ which is decided by the class of ($\Sigma^{0}_{1}$-definable) functions whose totality are proved in $T$.
In fact, we will give the exact proof-theoretic strength of $\RT^{2}_{2}$ 
by proving that $\RT^{2}_{2}+\WKLo$ 
is a $\Pi^{0}_{3}$-conservative extension of $\II$ (Theorem~\ref{thm-con-RT22}), where $\WKLo$ stands for weak K\"onig's lemma
and~$\IN$ is the $\Sigma^0_n$-induction scheme.
This answers the long-standing open question of determining the $\Pi^{0}_{2}$-consequences of $\RT^{2}_{2}$ or the consistency strength of $\RT^{2}_{2}$,
posed, e.g., in Seetapun and Slaman~\cite[Question~4.4]{Seetapun1995strength} Cholak, Jockusch and Slaman~\cite[Question~13.2]{CJS} Chong and Yang~\cite{CY2015} (see Corollaries~\ref{cor-answer1} and \ref{cor-answer-2}).
%and strengthens the above recent result by Chong, Slaman and Yang 
%since the $\Pi^{0}_{3}$-consequences of $\III$ are different from those of $\II$.
For this, we use a hybrid of forcing construction, indicator arguments,
and proof-theoretic techniques, and develop general tools simplifying
conservation results (see Theorems~\ref{thm-provably-large-conservative}, \ref{thm-g-Parsons} and~\ref{thm-con-Bn+1-vs-In}).
See Section~\ref{subsec:conservation} for the various studies of the proof-theoretic strength of $\RT^{2}_{2}$.
%Our conservation result gives actually a weak answer to the original question,
%namely, the proof-theoretic strength of $\RT^{2}_{2}$ is the same as $\BII$,
%since $\BII$ is a $\Pi^{0}_{3}$-conservative extension of $\II$.
Deciding the proof-theoretic strength of $\RT^{2}_{2}$ 
is also an important problem from a philosophical point of view.
In the sense of Simpson's partial realization \cite{partial} of Hilbert's Program, 
one would conclude that $\RT^{2}_{2}$ is \textit{finitistically reducible} (see Section~\ref{subsec:HP}).

\subsection{Reverse mathematics}\label{subsec:RM}

Reverse mathematics is a vast foundational program that seeks to determine 
which \emph{set existence axioms} are needed to prove theorems from ``ordinary'' mathematics.
It uses the framework of subsystems of second-order arithmetic.
Indeed, Friedman~\cite{MR0429508} realized that a large majority of theorems admitted a natural formulation
in the language of second-order arithmetic. The base theory $\RCAo$, standing for 
Recursive Comprehension Axiom, contains the basic axioms for first-order arithmetic
 (axioms of discrete ordered semi-ring) together with the~$\Delta^0_1$-comprehension
 scheme and the~$\Sigma^0_1$-induction scheme. $\RCAo$ can be thought of as capturing
\emph{computable mathematics}. 

Since then, thousands of theorems have been studied within the framework
of reverse mathematics. A surprising phenomenon emerged from the early years of
reverse mathematics: Most theorems studied require very weak axioms. 
Moreover, many of them happen to be \emph{equivalent} to one of five main sets of axioms,
that are referred to as the \emph{Big Five}, namely, $\RCAo$, weak K\"onig's lemma ($\WKLo$),
the arithmetic comprehension axiom ($\ACAo$), arithmetical transfinite recursion ($\ATRo$),
and $\Pi^1_1$-comprehension axiom ($\PCAo$). See Simpson~\cite{SOSOA} for an extensive study
of the Big Five and mathematics within them.
In this paper, we shall consider exclusively theorems which are provable
in $\ACAo$. See Hirschfeldt~\cite{Hirschfeldt2014Slicing} 
for a gentle introduction to the reverse mathematics below~$\ACAo$.

\subsection{Ramsey's theorem and its consequences}\label{subsect:ramsey-consequences}

Ramsey theory is a branch of mathematics 
studying the conditions under which some
structure appears among a sufficiently large collection of objects.
In the past two decades, Ramsey theory emerged as one of the most important topics in reverse mathematics.
This theory provides a large class of theorems escaping the Big Five phenomenon, 
and whose strength is notoriously hard to gauge.
Perhaps the most famous such theorem is \emph{Ramsey's theorem}.

\begin{defi}[Ramsey's theorem]
A subset~$H$ of~$\N$ is~\emph{homogeneous} for a coloring~$f : [\N]^n \to k$ (or \emph{$f$-homogeneous}) 
if all the $n$-tuples over~$H$ are given the same color by~$f$. 
$\RT^n_k$ is the statement ``Every coloring $f : [\N]^n \to k$ has an infinite $f$-homogeneous set''.
\end{defi}

Jockusch~\cite{Jockusch1972Ramseys} conducted a computational analysis of Ramsey's theorem,
later formalized by Simpson~\cite{SOSOA} within the framework of reverse mathematics.
Whenever $n \geq 3$, Ramsey's theorem for $n$-tuples happens to be equivalent to $\ACA$.
The status of Ramsey's theorem for pairs was open for decades, until Seetapun and Slaman~\cite{Seetapun1995strength} proved
that $\RT^2_2$ is strictly weaker than $\ACA$ over~$\RCAo$. Cholak, Jockusch and Slaman~\cite{CJS}
extensively studied Ramsey's theorem for pairs. On a computability-theoretic perspective,
every computable instance of~$\RT^n_k$ admits a $\Pi^0_n$ solution, while there exists
a computable instance of~$\RT^n_2$ with no $\Sigma^0_n$ solution~\cite{Jockusch1972Ramseys}. 
Ramsey's theorem for pairs is computationally weak in that it does not imply the existence of PA degrees~\cite{Liu2012RT22}, 
or any fixed incomputable set~\cite{Seetapun1995strength}.

In order to better understand the logical strength of~$\RT^2_2$,
Bovykin and Weiermann~\cite{BW} decomposed Ramsey's theorem for pairs
into the Erd\H{o}s-Moser theorem and the ascending descending sequence principle.
The Erd\H{o}s-Moser is a statement from graph theory.

\begin{defi}[Erd\H{o}s-Moser theorem] 
A tournament $T$ is an irreflexive binary relation such that for all $x,y \in \N$ with $x \not= y$, 
exactly one of $T(x,y)$ or $T(y,x)$ holds. A tournament $T$ is \emph{transitive} 
if the corresponding relation~$T$ is transitive in the usual sense. 
$\EM$ is the statement ``Every infinite tournament $T$ has an infinite transitive subtournament.''
\end{defi}

\begin{defi}[Ascending descending sequence]
Given a linear order (\textit{i.e.}, a transitive tournament)~$<_L$ on $\N$, an \emph{ascending} (\emph{descending}) sequence
is a set~$S$ such that for every~$x <_\N y \in S$, $x <_L y$ ($x >_L y$).
$\ADS$ is the statement ``Every infinite linear order admits an infinite ascending or descending sequence''.
\end{defi}

The Erd\H{o}s-Moser theorem provides together with the ascending descending principle an alternative decomposition of
Ramsey's theorem for pairs. Indeed, every coloring~$f : [\N]^2 \to 2$
can be seen as a tournament~$R$ such that~$R(x,y)$ holds if~$x < y$ and~$f(x,y) = 1$, or~$x > y$ and~$f(y, x) = 0$.
Then, $\EM$ is saying ``Every coloring $f : [\N]^n \to k$ has an infinite transitive subcoloring'' and
$\ADS$ is saying ``Every transitive coloring $f : [\N]^n \to k$ has an infinite $f$-homogeneous set''.
(In what follows, we always consider $\EM$ and $\ADS$ as these forms.)
We therefore obtain the following equivalence.
\begin{thm}[Hirschfeldt and Shore~\cite{HS2007}, Bovykin and Weiermann~\cite{BW}]
$\RCAo \vdash \RT^2_2 \leftrightarrow \ADS + \EM$.
\end{thm}

%Both of $\ADS$ and $\EM$ can be characterized by transitive coloring.
%A coloring $f:[\N]^{2}\to 2$ is said to be \emph{transitive on $X\subseteq \N$} if $f(a,b)=f(b,c)$ implies $f(a,b)=f(a,c)$ for any $a<b<c$ in $X$.
%\begin{thm}[Hirschfeldt/Shore~\cite{HS2007}, Bovykin/Weiermann~\cite{BW}]
%The following are provable within $\RCAo$.
%\begin{enumerate}
% \item $\ADS$ is equivalent to the statement that every transitive coloring $f : [\N]^2 \to 2$ has an infinite $f$-homogeneous set.
% \item $\EM$ is equivalent to the statement that for any coloring $f : [\N]^2 \to 2$, there exists an infinite set $X$ such that $f$ is %transitive on $X$.
%\end{enumerate}
%Therefore, $\RT^{2}_{2}$ is equivalent to $\ADS+\EM$ over $\RCAo$.
%\end{thm}

The ascending descending sequence has been introduced by Hirschfeldt
and Shore~\cite{HS2007}. They proved that $\ADS$ is strictly weaker than Ramsey's theorem for pairs.
On the other hand, Lerman, Solomon and Towsner~\cite{Lerman2013Separating} proved that the Erd\H{o}s-Moser theorem
is strictly weaker than $\RT^2_2$. For technical purposes, we shall consider a statement equivalent to the
ascending descending principle. Pseudo Ramsey's theorem for pairs has been introduced by Murakami, Yamazaki
and the second author~\cite{Murakami2014Ramseyan} to study a factorization theorem from automata theory.

\begin{defi}[Pseudo Ramsey's theorem for pairs]
A set~$H$ is \emph{pseudo-homogeneous} for a coloring $f : [\N]^2 \to k$
if there is a color~$c < k$ such that every pair $\{x, y\} \in [H]^2$
are the endpoints of a finite sequence $x_0 < x_1 < \dots < x_n$
such that~$f(x_i, x_{i+1}) = c$ for each~$i < n$.
$\psRT^2_k$ is the statement ``Every coloring $f : [\N]^2 \to k$ has an infinite $f$-pseudo-homogeneous set''.
\end{defi}

In particular, if $f : [\N]^2 \to 2$ is a transitive coloring,
then any set~$H$ pseudo-homogeneous for~$f$ is already homogeneous for~$f$.
Thus, $\RCAo + \psRT^2_2$ implies $\ADS$ (see~\cite{Murakami2014Ramseyan}).
The first author~\cite{Patey2016reverse} and Steila (see~\cite{SY201X}) independently 
proved the reverse implication, namely, $\RCAo + \ADS$ implies~$\psRT^2_2$.%

%\keita{Can you add a short explanation about the computability-theoretic strength of $\RT^{2}_{2}$?}

\subsection{Proof strength and conservation results}\label{subsec:conservation}

In the study of reverse mathematics, deciding the first-order or proof-theoretic strength 
of axioms and mathematical principles is one of the main topics.
This is usually analyze through the conservation theorems.
Especially, the conservation result for weak K\"onig's lemma always plays the central role
as a large part of mathematics can be proven within $\WKLo$ (see Simpson~\cite{SOSOA}).
The following theorems show that one can use weak K\"onig's lemma almost freely to seek for first-order consequences.
\begin{thm}[Friedman\cite{friedmancom}, see \cite{SOSOA}]\label{thm-Friedman-conservation}
$\WKLo$ is a $\Pi^{0}_{2}$-conservative extension of $\PRA$.
\end{thm}
\begin{thm}[Harrington, see \cite{SOSOA}]\label{thm:harrington-conservation-wkl}
$\WKLo$ is a $\Pi^{1}_{1}$-conservative extension of $\II$.
\end{thm}

The $\Sigma^0_2$-bounding principle ($\BII$) informally asserts that a finite union of finite sets is finite.
Many mathematical reasonings make an essential use of $\BII$ and in particular $\RT^{2}_{2}$ implies $\BII$.
The strength of the $\Sigma^0_2$-bounding principle is therefore important for the study of combinatorial principles.
\begin{thm}[H\'ajek\cite{Hajek}]
$\WKLo+\BII$ is a $\Pi^{1}_{1}$-conservative extension of $\BII$.
\end{thm}

Thankfully, $\BII$ can be freely used for a restricted class of formulas.
Let $\tilde{\Pi}{}^{0}_{3}$ be a class of formulas of the form $\A X\varphi(X)$ 
where $\varphi$ is a $\Pi^{0}_{3}$-formula. The following is a parameterized 
version of the Parsons, Paris and Friedman conservation theorem.
\begin{thm}[see, e.g.,~\cite{handbook-of-proof} or \cite{Kaye}]\label{thm:bsig2-pi03-conservative-isig1}
$\BII$ is a $\tilde{\Pi}{}^{0}_{3}$-conservative extension of $\II$.
\end{thm}
Note that the above four conservation theorems are frequently used in this paper, and so we shall not mention them explicitely.

About the first-order/proof-theoretic strength of Ramsey's theorem for pairs, there are long series of studies by various people and various methods.
Hirst~\cite{Hirst-PhD} showed that $\RT^{2}_{2}$ implies the $\Sigma^0_2$-bounding principle ($\BII$).
On the other hand, Cholak, Jockusch and Slaman~\cite{CJS} showed that $\WKLo+\RT^{2}_{2}+\III$ 
is a $\Pi^{1}_{1}$-conservative extension of $\III$.
Thus, the first-order strength of Ramsey's theorem for pairs and two colors
 is in between $\BII$ and $\III$.
%After this work, the project of deciding the first-order strength of $\RT^{2}_{2}$ 
%has been strongly carried out using forcing constructions or priority arguments 
%on nonstandard models of $\BII$, mainly by Chong, Slaman and Yang~\cite{CSY2012,CSY2014}.
%They proved in particular that $\WKLo+\RT^{2}_{2}$ does not imply $\III$~\cite{CSY201X}.
After this work, many advanced studies are done to investigate the first-order strength of Ramsey's theorem and related combinatorial principles.
One of the most important methods for these studies consists in adapting computability-theoretic techniques for combinatorial principles.
By this method, Chong, Slaman and Yang~\cite{CSY2012} showed that two weaker combinatorial principles, namely, 
the ascending descending sequence ($\ADS$) and the chain antichain principle ($\CAC$),
 introduced by Shore and Hirschfeldt~\cite{HS2007}, are $\Pi^{1}_{1}$-conservative over $\BII$.
In~\cite{CSY2014}, they showed that $\WKLo+\SRT^{2}_{2}$ does not imply $\III$, 
and they improved the result and proved that $\RT^{2}_{2}$ does not imply $\III$ in~\cite{CSY201X}. 
More recently, Chong, Kreuzer and Yang [unpublished] showed that $\WKLo+\SRT^{2}_{2}$ is $\Pi^{0}_{3}$-conservative over $\RCAo+\WF(\omega^{\omega})$, where~$\WF(\omega^{\omega})$ asserts the well-foundedness of $\omega^{\omega}$.
%So, the rest of the question is whether the first-order part of $\RT^{2}_{2}$ is strictly stronger than $\BII$ or not.

Besides the computability-theoretic techniques, many other significant approaches can be found in the literature. 
Kohlenbach and Kreuzer~\cite{KK2009} and Kreuzer~\cite{Kreuzer2012} 
characterized the $\Pi^{0}_{2}$-parts of $\RT^{2}_{2}$ and $\CAC$ with several different 
settings by proof-theoretic approaches. Bovykin and Weiermann~\cite{BW} and the second 
author~\cite{Y-RIMS2013} showed that indicators defined by Paris's density notion 
can approach the proof-theoretic strength of various versions of Ramsey's theorem, 
and by a similar method, the second author~\cite{Y-MLQ13} also showed that 
$\RT^{n}_{k}+\WKLs$ is fairly weak and is a $\Pi^{0}_{2}$-conservative extension of $\RCAs$,
where~$\RCAs$ is $\RCAo$ with only $\Sigma^0_0$-induction and the exponentiation.
There are also many studies of the proof-theoretic strength of Ramsey's theorem by using ordinal analysis, by Kotlarski, Weiermann, et al.~\cite{Weiermann2004, KPW, BK99, BK02, BK06, DW}.

Moreover, the study of proof-theoretic strength of Ramsey's theorem for pairs has a solid connection to computer science.
Especially, in the field of termination analysis, Podelski and Rybalchenko~\cite{PR2004} introduced a new method
to verify the termination of programs by using Ramsey's theorem for pairs, and based on this method,
many termination verifiers are invented.
On the other hand, as we can see in Buchholz \cite{Buchholz},
 it is known that proof theory can provide an upper bound for termination proofs
since the termination statement is always described by a $\Pi^{0}_{2}$-formula.
In fact, the termination theorem argued in \cite{PR2004} is essentially equivalent to a weaker version of Ramsey's theorem for pairs, and the proof-theoretic strength of Ramsey's theorem can give a general upper bound for all of those types of termination proofs. See \cite{SY201X}.
%the study of the $\Pi^{0}_{2}$-parts of Ramsey-type theorems has some application to computer science~\cite{SY201X}.

\subsection{Second-order structures of arithmetic and their cuts}
A structure for the language of second-order arithmetic $\mathcal{L}_{2}$ is a pair $(M,S)$ 
where $M=(M, +_{M}, \cdot_{M},\allowbreak 0_{M}, 1_{M}, <_{M})$ 
is a structure for the language of first-order (Peano) arithmetic $\mathcal{L}_{\PA}$, and $S$ is a subset of the power set of $M$.
%If there is no confusion, we usually indicate the pair $(M,S)$ for the second-order structure.

\begin{defi}[Cut]
Given a structure~$M$ of the first-order arithmetic, a substructure $I \subseteq M$ is said to be a \emph{cut} of $M$ 
(abbreviated $(I\subseteq_{e} M)$) if $I=\{a\in M\mid \E b\in I(a<_{M}b)\}$.
\end{defi}

Here, the standard first-order structure $\omega$ can be considered as the smallest cut of any first-order structure.
Given a structure $(M,S)$, a cut $I\subseteq_{e} M$ induces the second-order structure $(I,S\rest I)$,
where $S\rest I:=\{X\cap I\mid X\in S\}$.
We sometimes consider $S$ as a family of unary predicates on $M$ and identify $(M,S)$ as an $\mathcal{L}_{\PA}\cup S$-structure.
Accordingly, $(I,S\rest I)$ can be considered as an $\mathcal{L}_{\PA}\cup S$-substructure of $(M,S)$.
Note that $S\rest I$ may then be a multiset on $I$, but this is harmless without second-order equality.
In this sense, one can easily check that $(I,S\rest I)$ is a $\Sigma^{0}_{0}$-elementary substructure of $(M,S)$ if $I$ is closed under $+_{M}$ and $\cdot_{M}$.

We write $\Cod(M)$ for the collection of all $M$-finite subsets, \textit{i.e.}, 
subsets of $M$ canonically coded by elements in $M$ (as the usual binary expansion).
Given some cut $I\subseteq_{e} M$, we write $\Cod(M/I)$ for $\Cod(M)\rest I$.
If $I\subsetneq M$, then $\Cod(M/I)=S\rest I$ for any $S\subseteq\mathcal{P}(M)$ such that $(M,S)\models \RCAo$, thus $(I, S\rest I)$ only depends on $M$ and $I$.

A cut $I\subseteq_{e} M$ is said to be \emph{semi-regular} if $I\cap X$ is bounded for any $X\in \Cod(M)$ such that $|X|\in I$, where $|X|$ denotes the cardinality of $X$ in $M$.
A semi-regular cut is one of the central notions in the study of cuts, especially with the connection to second-order structures, since it characterizes the models of $\WKLo$.
We will use the following theorem throughout this paper without mentioning it explicitly.

\begin{thm}[see, e.g., Theorems 7.1.5, 7.1.7 of \cite{Kossak-Schmerl}]
Let $I$ be a cut of a first-order structure $M$. Then, $I$ is semi-regular if and only if $(I,\Cod(M/I))\models \WKLo$. 
\end{thm}

Bounding principles are also characterized by cuts with some elementarity condition.
In this paper, we will use the following characterization.
\begin{thm}[Proposition 3 of Clote \cite{Clote1985}, see also Paris and Kirby~\cite{Paris-Kirby}]\label{thm-elem-cut-vs-Bn}
Let $n\ge 1$.
Let $(M,S)$ be a model of $\IN[n-1]$, and let $I$ be a cut of $M$ such that $(I, S\rest I)$ be a $\Sigma^{0}_{n}$-elementary substructure of $(M,S)$.
Then, $(I,S\rest I)\models\BN[n+1]$.
\end{thm}

\subsection{Hilbert's program and finitistic reductionism}\label{subsec:HP}

During the early 20th century, mathematics went through a serious foundational crisis,
with the discovery of various paradoxes and inconsistencies. Some great mathematicians such as Kronecker, Poincar\'e and Brouwer
challenged the validity of infinitistic reasoning. Hilbert~\cite{hilbert1926} proposed a three-step program 
to answer those criticisms. First, he suggested to identify the finitistic part of mathematics, then to axiomatize infinite reasoning,
and eventually to give a finitistically correct consistency proof of this axiomatic system. 
However, his program was nipped in the bud by G\"odel's incompleteness theorems~\cite{MR2260656}. 

In 1986, Simpson~\cite{partial} proposed a formal interpretation of Hilbert's program
by taking primitive recursive arithmetic (PRA) as the base system for capturing finitistic reasoning.
This choice was convincingly justified by Tait~\cite{tait1981finitism}. Simpson took
second-order arithmetic ($Z_2$) as the big system capturing infinitistic reasoning, based on the work
of Hilbert and Bernays~\cite{MR3027390}. In this setting, finitistic reductionism can be interpreted
as proving that $Z_2$ is $\Pi^0_1$-conservative over~PRA. By G\"odel incompleteness theorem,
this cannot be the case. However, Simpson proposed to consider $\Pi^0_1$-conservation of subsystems of second-order arithmetic
over PRA as a partial realization of Hilbert's program. He illustrated his words with weak K\"onig's lemma ($\wkl$)
which was proven by Friedman to be $\Pi^0_2$-conservative over~PRA (Theorem~\ref{thm-Friedman-conservation}).
In this paper, we contribute to Hilbert's program by showing that $\WKL + \RT^2_2$ is $\Pi^0_2$-conservative over~PRA,
and therefore that $\RT^2_2$ is finitistically reducible. Moreover, we develop general tools to simplify the proofs of $\Pi^0_3$-conservation theorems, and thereby contribute to the simplification of the realization of Hilbert's program.

\subsection{Notation}

In order to avoid confusion between the theory and the meta-theory, we shall use
$\omega$ to denote the set of (standard) natural numbers, and $\N$ to denote the sets
of natural numbers inside the system. Accordingly, we shall write $\bbomega$
for the ordinal $\omega$ in the system.
We write $[a,b]_{\N}, (a,b)_{\N},(a,\infty)_{\N},\dots$ for intervals of natural numbers, e.g., $(a,b]_{\N}=\{x\in\N\mid a<x\le b\}$.
Given a set~$X$ and some~$n \in \N$, $[X]^n$ is the collection of all sets of size~$n$.
$[X]^{<\N}$ is the union $\bigcup_n [X]^n$.

We use $\Pi_{n}$, $\Sigma_{n}$, $\Delta_{n}$ to denote first-order formulas without set parameters,
whereas $\Pi^{0}_{n}$, $\Sigma^{0}_{n}$, $\Delta^{0}_{n}$ are second-order formulas, \textit{i.e.}, with set parameters.
A $\tilde{\Pi}{}^{0}_{n}$-formula is a second-order formula of the form $(\forall X)\varphi(X)$ where $\varphi$ is a $\Pi^0_n$-formula.

Given two sets~$A, B$, $A\oplus B=\{2x\mid x\in A\}\cup\{2x+1\mid x\in B\}$, $A \subseteq_{\fin} B$ means that $A$ is a finite subset of~$B$, and $A \subseteq^{*} B$ means that the set $A$ is included, up to finite changes, in~$B$.
We write~$A < B$ for the formula $(\forall x \in A)(\forall y \in B)x < y$.
Whenever~$A = \{x\}$, we shall simply write~$x < B$ for~$A < B$.
A set~$X$ can be seen as an infinite join~$X = \bigoplus_i X_i$, where $x \in X_i$ iff $\langle i,x \rangle \in X$.
We then write~$X[i]$ for~$X_i$.
Given a set~$X$ or a string~$\sigma$ and some integer~$m \in \omega$, 
we write~$X \rest m$ for the initial segment of~$X$ (resp.\ $\sigma$) of length~$m$.

\subsection{Structure of this paper}

The main target of this paper is the following conservation theorem.
\begin{theorem*}
$\WKLo+\RT^{2}_{2}$ is a $\tilde{\Pi}{}^{0}_{3}$-conservative extension of $\II$.
\end{theorem*}
We will prove this in the following way.

In Sections~\ref{section-large-and-dense} and \ref{section-density}, we will explain that
$\tilde{\Pi}{}^{0}_{3}$-consequences of Ramsey's theorem and its variations 
are characterized by some largeness notions of finite sets.
We will introduce largeness notion for $\Gamma$, where $\Gamma$ is any of $\RT^{2}_{2}$, 
$\psRT^{2}_{2}$ (which is equivalent to $\ADS$), and $\EM$.
Roughly speaking, giving a bound for largeness for $\Gamma$ within $\II$ provides 
$\tilde{\Pi}{}^{0}_{3}$-conservation for $\WKLo+\Gamma$ over $\II$ (Theorem~\ref{thm-provably-large-conservative}).

According to the decomposition of $\RT^{2}_{2}$ into $\ADS$ and $\EM$ and the amalgamation theorem (Theorem~\ref{thm:amalgamation}), 
the conservation for $\RT^{2}_{2}$ can be decomposed into the conservation for 
$\ADS$ and the conservation for $\EM$.
In Section~\ref{section-ADS}, we give a bound for the largeness notion for $\psRT^{2}_{2}$ 
(Lemma~\ref{lem-ADS-largeness}) by using the Ketonen/Solovay theorem.
It actually provides the conservation result for $\ADS$ (Corollary~\ref{cor-con-ADS}).

It is rather complicated to give a bound for the largeness notion for $\EM$.
For this, we will introduce a new combinatorial principle called the \emph{grouping principle}. 
We mainly focus on the grouping principle for pairs and two colors $\GP$.
Section~\ref{section-GP} is devoted to the reverse mathematical study of $\GP$, 
especially from the view point of computability theory.
In Section~\ref{section-GP-con}, we will prove a conservation theorem for $\GP$ (Theorem~\ref{thm:GP-conservation}).
For this, we will modify the construction of a low solution for the stable version of $\GP$ (Theorem~\ref{thm:grouping-low-solution}) 
presented in the previous section.

In Section~\ref{section-RT22}, we give a bound for the largeness notion for $\EM$ (Lemma~\ref{lem-EM-largeness})
 by using a finite version of the grouping principle, which is a consequence of $\GP$.
It provides the conservation result for $\EM$ (Theorem~\ref{thm-con-EM}).
Then, combining this with the conservation result in Section~\ref{section-ADS} by the amalgamation theorem,
 we obtain the main theorem.

The main theorem can be formalized within $\WKLo$, and that leads to the consistency equivalence of $\II$ and $\RT^{2}_{2}$.
This is argued in Section~\ref{section-consistency-strength}.

\section{Largeness}\label{section-large-and-dense}

A family of finite sets of natural numbers $\LL\subseteq [\N]^{<\N}$ is said 
to be a \textit{largeness notion} if any infinite set has a finite subset in $\LL$
and $\LL$ is closed under supersets. A finite set $X\in [\N]^{<\N}$ is said to be $\LL$-large 
if $X\in \LL$. A (possibly largeness) notion $\LL\subseteq [\N]^{<\N}$ is said to be \emph{regular} 
if for any $\LL$-large set~$F$, any finite set~$G\subseteq_{\fin}\N$ for which there exists an order-preserving 
injection $h:F\to G$ such that $(\forall x \in F)h(x)\le x$, then $G \in \LL$.
A $\Delta_{0}$-definable notion $\LL\subseteq[\N]^{<\N}$ is said to be a
(regular) largeness notion provably in $\II$ if $\II\vdash$``$\LL$ is a (regular) largeness notion''.
The idea of a largeness notion is introduced in Aczel~\cite{Aczel1981} (it is called `density' in \cite{Aczel1981}).
In this paper, we shall mainly consider regular largeness notions provably in $\II$.
\begin{exa}
The family~$\LL_\bbomega = \{ X \subseteq_{\fin} \N : |X| > \min X \}$ is a regular largeness notion provably in $\II$.
\end{exa}

The notion of largeness enjoys a property similar to the pigeonhole principle,
as states the following lemma.

\begin{lem}[$\WKLo+\BII$]\label{lem-largeness}
For any largeness notion $\LL$, for any infinite set $X$ 
and for any $k, N_{0}\in\N$, there exists $N_{1}\in\N$ such that for any partition 
$X\cap[N_{0},N_{1}]_{\N}=X_{0}\sqcup\dots\sqcup X_{k-1}$, there exists an $\LL$-large set $F$ 
such that $F\subseteq X_{i}$ for some $i< k$.
\end{lem}
Since $\WKLo+\BII$ is a $\tilde{\Pi}^0_3$-conservative extension of~$\II$, 
this lemma for a $\Delta_{1}$-definable largeness notion provably in $\II$ is provable in $\II$.
\begin{proof}
By $\BII$, for any partition $X\cap [N_{0},\infty)_{\N}=X_{0}\sqcup\dots\sqcup X_{k-1}$, one of the $X_{i}$'s is infinite and thus it contains an $\LL$-large subset.
Thus, a bound for such an $\LL$-large set can be obtained by the usual compactness argument which is available within $\WKLo$.
\end{proof}
\begin{rem}\label{rem:on-lem-largeness}
Note that the use of $\BII$ in the previous lemma is essential.
If $\BII$ fails, there exists a partition $X=X_{0}\sqcup\dots\sqcup X_{k-1}$ such that each of the $X_{i}$'s is finite.
Then, $\LL=\{F\in [\N]^{<\N}\mid \A i<k(F\not\subseteq X_{i})\}$ is a largeness notion failing the lemma.
The use of~$\WKLo$ is not essential since the argument can be formalized within~$\RWKL+\BII$,
where~$\RWKL$ denotes the Ramsey-type weak K\"onig's lemma introduced by Flood~\cite{Flood2012Reverse}.
 \end{rem}

\subsection{$\alpha$-largeness}

From now on, we fix a primitive recursive notation for ordinals below $\epsilon_{0}$.
In this paper, we actually use ordinals of the form $\alpha=\sum_{i<k}\bbomega^{n_{i}}<\bbomega^{\bbomega}$ 
where $n_{i}\in\N$ and $n_{0}\ge\dots\ge n_{k-1}$. (We write $1$ for $\bbomega^{0}$, and $\bbomega^{n}\cdot k$ for $\sum_{i<k}\bbomega^{n}$.)
For a given $\alpha<\bbomega^{\bbomega}$ and $m\in \N$, define $0[m]=0$, $\alpha[m]=\beta$ 
if $\alpha=\beta+1$ and $\alpha[m]=\beta+\bbomega^{n-1}\cdot m$ if $\alpha=\beta+\bbomega^{n}$ for some $n\ge 1$.

\begin{defi}[$\II$]
Let $\alpha<\bbomega^{\bbomega}$. A set $X = \{ x_0 < \dots < x_{\ell-1} \}\subseteq_{\fin}\N$ 
is said to be \emph{$\alpha$-large} if $\alpha[x_{0}]\dots[x_{\ell-1}]=0$. 
In other words, any finite set is $0$-large, and $X$ is said to be $\alpha$-large if
\begin{itemize}
 \item $X\setminus \{\min X\}$ is $\beta$-large if $\alpha=\beta+1$,
 \item $X\setminus \{\min X\}$ is $(\beta+\bbomega^{n-1}\cdot\min X)$-large if $\alpha=\beta+\bbomega^{n}$.
\end{itemize}
We let $\LL_{\alpha}=\{X\subseteq_{\fin} \N\mid X$ is $\alpha$-large$\}$.
\end{defi}

In particular, a set~$X$ is $m$-large iff $|X| \geq m$ and $\bbomega$-large iff $|X| > \min X$.
See~\cite{Ha-Pu} for the general definition of $\alpha$-largeness.
One can easily see that if $X\subseteq Y$ for some $\alpha$-large set~$X$ and some finite set~$Y$, 
then $Y$ is $\alpha$-large.

We say that $X$ is \emph{$\alpha$-small} if it is not $\alpha$-large.
The following basic combinatorics have been proven in~\cite[Theorem~II.3.21]{Ha-Pu} in their full generality.

\begin{lem}[$\II$]\label{lem-alpha-large-sum}\label{lem-alpha-small-sum}
Fix any $k,n\in\N$.
\begin{itemize}
	\item[(i)] A set $X$ is $\bbomega^{n}\cdot k$-large if and only 
	if it is a union of some $\bbomega^{n}$-large finite sets $X_{0} < \dots < X_{k-1}$.
	\item[(ii)] A set $X$ is $\bbomega^{n}\cdot k$-small 
	if it is a union of $\bbomega^n$-small finite sets $X_{0} < \dots < X_{k-1}$.
\end{itemize}
\end{lem}

In particular, $\{k\}\cup X_{0}\cup\dots\cup X_{k-1}$ is $\bbomega^{n+1}$-large 
if each of $X_{i}$ is $\bbomega^{n}$-large and $k<X_{0} < \dots < X_{k-1}$.
Similarly, if $\{k\}\cup X_{0}\cup\dots\cup X_{k-1}$ is $\bbomega^{n+1}$-large 
and $k<X_{0} < \dots < X_{k-1}$, then one of $X_{i}$'s is $\bbomega^{n}$-large.

The following theorem corresponds to the well-known fact that the proof-theoretic ordinal of $\II$ is $\omega^{\omega}$.

\begin{thm}\label{thm-omega-largeness}
For any $n\in\omega$, $\II$ proves that $\LL_{\bbomega^{n}}$ is a regular largeness notion.
\end{thm}
\begin{proof}
 One can easily check the regularity within $\II$.
We will see that $\II\vdash$ ``any infinite set contains an $\bbomega^{n}$-large subset'' by (external) induction.
The case $n=0$ is trivial.
We show the case $n=k+1$.
Within $\II$, let an infinite set $X$ be given.
Then, by the induction hypothesis and $\Sigma^{0}_{1}$-induction, one can find $\min X$-many $\bbomega^{k}$-large sets $F_{i}\subseteq X$
such that $\min X<F_{0} < \dots < F_{\min X-1}$.
By Lemma~\ref{lem-alpha-large-sum} and the discussion below it, $\{\min X\}\cup F_{0}\cup\dots\cup F_{\min X-1}$ is $\bbomega^{k+1}$-large.
\end{proof}

\subsection{Largeness for Ramsey-like statements}

Many Ramsey-type theorems studied in reverse mathematics
are statements of the form ``For every coloring $f: [\N]^n \to k$,
there is an infinite set~$H$ satisfying some structural properties''.
The most notable example is Ramsey's theorem, which asserts for every coloring $f : [\N]^n \to k$
the existence of an infinite $f$-homogeneous set.
These statements can be seen as \emph{mathematical problems},
whose \emph{instances} are coloring, and whose \emph{solutions}
are the sets satisfying the desired structural properties.

\begin{definition}\label{def-R-like-and-rest}
A \emph{Ramsey-like-$\Pi^1_2$-formula} is a $\Pi^1_2$-formula of the form
$$
	(\forall f:[\N]^{n}\to k)(\exists Y)(Y\mbox{ is infinite}\wedge \Psi(f,Y))
$$
where~$n,k\in\omega$ and~$\Psi(f,Y)$ is of the form $(\A G\subseteq_{\fin}Y)\Psi_{0}(f\rest[[0,\max G]_{\N}]^{n},G)$ such that~$\Psi_{0}$ is a $\Delta^{0}_{0}$-formula.
\end{definition}

In particular, $\RT^n_k$ is a Ramsey-like-$\Pi^1_2$-statement
where~$\Psi(f,Y)$ is the formula ``$Y$ is homogeneous for~$f$''.
Similarly, $\psRT^{2}_{k}$ and~$\EM$ are Ramsey-like-$\Pi^1_2$ statements.
On the other hand, $\SRT^{2}_{2}$ is not a Ramsey-like-$\Pi^1_2$-statement. However, $\SRT^{2}_{2}$
is equivalent to the Ramsey-like-$\Pi^1_2$-formula saying ``for any $2$-coloring $f$ on $[\N]^{2}$, there exists an infinite set $Y$ such that $Y$ is homogeneous for $f$ or there exists $a<\min Y$ witnessing the non-stability of~$f$, that is, such that for any $x,y\in Y$ there exist $b,c\in [x,y)_{\N}$ such that $f(a,b)\neq f(a,c)$''.
Although the definition of a Ramsey-like-$\Pi^1_2$-formula seems very restrictive,
we can show that it entails a much larger class of $\Pi^1_2$-statements.

\begin{definition}
A \emph{restricted-$\Pi^1_2$-formula} is a $\Pi^1_2$-formula of the form 
$\forall X\exists Y\Theta(X,Y)$ where $\Theta$ is a $\Sigma^{0}_{3}$-formula.
\end{definition}

\begin{prop}\label{prop-R-like-and-rest-formulas}
For any restricted-$\Pi^{1}_{2}$-formula $\A X\E Y\Theta(X,Y)$, 
there exists a Ramsey-like-$\Pi^{1}_{2}$-formula $\A X\E Z(Z\mbox{ is infinite}\wedge \Psi(X,Z))$ such that 
 $$\WKLo\vdash \A X(\E Y\Theta(X,Y)\leftrightarrow \E Z(Z\mbox{ is infinite}\wedge \Psi(X,Z))).$$
(Here, $X$ is considered as a function $X:[\N]^{1}\to 2$ in the definition 
of Ramsey-like-$\Pi^{1}_{2}$-formula.)
\end{prop}
\begin{proof}
We work within $\WKLo$.
Let $\A X\E Y\Theta(X,Y)$ be a restricted-$\Pi^{1}_{2}$-formula.
Without loss of generality, one can write $\Theta(X,Y)\equiv \A n \E m\theta(X\rest m,Y\rest m,n,m)$ where $\theta$ is $\Sigma^{0}_{0}$ since existential number quantifier can be replaced with an existential set quantifier.
%Given a finite set $F\subseteq\N$ and some $n< |F|$, let $p(F,n)$ be the $(n+1)$-st smallest element of $F$.
%\ludovic{I do not understand why you use $p(F,n)$ instead of $max(F)$.}
%\keita{I agree $\max F$ is enough. Actually, I remember I first used $\max F$ and then changed it to $p(F,n)$, but I cannot remember the reason.
%Probably, I misunderstood something at that time.}
Define a formula $\Psi(X,Z)$ as follows:
\begin{align*}
%\Psi(X,Z)\equiv \A F\subseteq_{\fin} Z(F\neq \emptyset \to (&\E \sigma\in 2^{\max F})(\A n<|F|)(\E m<p(F,n))\theta(X[m],\sigma\rest m,n,m)).
\Psi(X,Z)\equiv \A F\subseteq_{\fin} Z(F\neq \emptyset \to (&\E \sigma\in 2^{\max F})(\A n<|F|)(\E m<\max F)\theta(X\rest m,\sigma\rest m,n,m)).
\end{align*}
We now show that $\E Y\Theta(X,Y)\leftrightarrow \E Z(Z\mbox{ is infinite}\wedge \Psi(X,Z))$.

To show the left to right implication, take $Y\subseteq\N$ such that $\A n \E m \theta(X\rest m, Y\rest m,n, m)$.
Define an infinite increasing sequence $\langle z_{i}\mid i\in\N \rangle$ as $z_{0}=\min \{m+1\mid \theta(X\rest m, Y\rest m,0,m)\}$,
and $z_{i+1}=\min \{m>z_{i}\mid \theta(X\rest m, Y\rest m,i+1,m)\}$.
Let $Z=\{z_{i}\mid i\in\N\}$, then we have $Z\mbox{ is infinite}\wedge \Psi(X,Z)$ (given a non-empty set $F\subseteq_{\fin}Z$, 
set $\sigma=Y\rest\max F$).

To show the right to left implication, take $Z\subseteq\N$ such that $Z\mbox{ is infinite}\wedge \Psi(X,Z)$.
Define a tree $T$ as 
\begin{align*}
%T=\{\sigma\in 2^{<\N}\mid \A F\subseteq Z\cap[0, |\sigma|](F\neq\emptyset\to (\A n<|F|)(\E m<p(F,n))\theta(X[m],\sigma\rest m, k,n,m))\}. 
T=\{\sigma\in 2^{<\N}\mid \A F\subseteq Z\cap[0, |\sigma|]_{\N}(F\neq\emptyset\to (\A n<|F|)(\E m<\max F)\theta(X\rest m,\sigma\rest m, k,n,m))\}. 
\end{align*}
Then, since $\Psi(X,Z)$ holds, the tree $T$ is infinite.
By $\WKLo$, take $Y\in [T]$. One can check that $\A n \E m\theta(X\rest m, Y\rest m,n,m)$ since $Z$ is infinite.
\end{proof}

\begin{defi}[$\II$]\label{def-Gamma-large}
Fix an ordinal $\alpha < \bbomega^\bbomega$ and a Ramsey-like-$\Pi^1_2$-statement
$\Gamma \equiv (\forall f:[\N]^{n}\to k)(\exists Y)(Y\mbox{ is infinite}\wedge \Psi(f,Y))$.
%where~$n,k\in\omega$ and~$\Psi(X,Y)$ is of the form $(\A G\subseteq_{\fin}Y)\Psi_{0}(X,G)$ such that~$\Psi_{0}$ is a $\Delta^{0}_{0}$ formula..
A set $Z\subseteq_{\fin}\N$ is said to be $\alpha$-large($\Gamma$) if
% for any~$f:[[0,\max Z]_{\N}]^{n}\to k$, 
%there exists an $\alpha$-large set~$Y \subseteq Z$ such that~$\Psi(f, Y)$ holds.
for any~$f:[[0,|Z|)_{\N}]^{n}\to k$,
there is an $\alpha$-large set~$Y \subseteq Z$
such that~$\Psi(f, p_{Z}(Y))$ holds, where $p_{Z}$ is the unique order preserving bijection from $Z$ to $[0,|Z|)_{\N}$.
\end{defi}
By the definition of a Ramsey-like-$\Pi^1_2$-formula, if $Z'\supseteq Z$ and $Z$ is $\alpha$-large($\Gamma$), then $Z'$ is $\alpha$-large($\Gamma$).
For the usual Ramsey type statements we consider ($\RT^{n}_{k}$, $\psRT^{2}_{k}$, $\EM$, \dots),
we usually identify a function $f:[[0,|Z|)_{\N}]^{n}\to k$ with $f\circ (p_{Z})^{n}:[Z]^{n}\to k$ and just discuss on $[Z]^{n}$.

\begin{exa}
A set~$Z$ is $\alpha$-large($\RT^n_k$) if for every coloring $f : [Z]^n \to k$,
there is an $\alpha$-large $f$-homogeneous set~$Y \subseteq Z$.
\end{exa}

Note that~$\alpha$-largeness($\Gamma$) for $\Gamma \in \{\RT^n_k, \psRT^{2}_{k}, \EM\}$
are all $\Delta_0$-definable notions, and $\II$ proves that they are all regular.
However, it is not obvious within $\II$ that they are all largeness notions.
Actually, showing within $\II$ that $\alpha$-largeness($\RT^2_2$) is a largeness notion
is the key to know the $\Pi^{0}_{3}$-part of $\RT^{2}_{2}$.
Our approach is to measure the size of $\alpha$-large($\Gamma$) sets by comparing them with $\alpha$-large sets.
The following classical theorem is fundamental for this purpose.
(It is not hard to check that the proof works within $\II$.)
%We let $\LL_{\alpha}=\{X\subseteq \N\mid X$ is $\alpha$-large$\}$, 
%and let $\LL^{\Gamma}_{\alpha}=\{X\subseteq \N\mid X$ is 
%$\alpha$-large($\Gamma$)$\}$, where $\Gamma\in \{\RT^{n}_{k}, \ADS, \EM\}$.
%Showing that $\II$ proves $\LL^{\ADS}_{\bbomega^{k}}$, $\LL^{\EM}_{\bbomega^{k}}$ 
%and $\LL^{\RT^{2}_{2}}_{\bbomega^{k}}$ are all largeness notions is the key 
%to know the $\Pi^{0}_{3}$-part of $\RT^{2}_{2}$.

\begin{thm}[{Ketonen and Solovay\cite[Section 6]{KS81}}]\label{SK-theorem}
Let $k\in\omega$.
The following is provable within $\II$. If a finite set $X$ is $\bbomega^{k+4}$-large 
and $\min X>3$, then $X$ is $\bbomega$-large($\RT^{2}_{k}$).
\end{thm}

\section{Density and $\tilde{\Pi}{}^{0}_{3}$-conservation}\label{section-density}

The goal of this section is to prove the following theorem.

\begin{theorem}[Conservation through largeness]\label{thm-provably-large-conservative}
Let $\Gamma$ be a Ramsey-like-$\Pi^{1}_{2}$-statement.
If $\bbomega^{k}$-large$(\Gamma)$ness is a largeness notion provably in $\II$ for any $k\in\omega$,
then, $\WKLo+\Gamma$ is a $\tilde{\Pi}{}^{0}_{3}$-conservative extension of $\II$.
\end{theorem}

For this, we will introduce an iterated version of a largeness notion which is called ``density''.
This notion is introduced by Paris in \cite{Paris1978}.
%Originally, it is a finite iteration of Paris/Harrington statement.

\begin{defi}[$\II$, Density notion]\label{def-density-with-indicator}
Fix a Ramsey-like $\Pi^1_2$-statement
$$
  \Gamma = (\forall f:[\N]^{n}\to k)(\exists Y)(Y\mbox{ is infinite}\wedge \Psi(f,Y)).
$$
We define the notion of~\textit{$m$-density($\Gamma)$} of a finite set~$Z\subseteq\N$ inductively as follows.
First, a set~$Z$ is \textit{$0$-dense($\Gamma$}) if it is $\bbomega$-large and $\min Z>1$.
Assuming the notion of~$m$-density($\Gamma$) is defined, 
a set~$Z$ is \textit{$(m+1)$-dense($\Gamma$)} if
\begin{itemize}
 \item for any~$f:[[0,|Z|)_{\N}]^{n}\to k$,
there is an $m$-dense($\Gamma$) set~$Y \subseteq Z$
such that~$\Psi(f, p_{Z}(Y))$ holds, where $p_{Z}$ is the unique order preserving bijection from $Z$ to $[0,|Z|)_{\N}$, and,
 \item for any partition $Z_{0}\sqcup\dots \sqcup Z_{\ell-1}=Z$ such that $\ell\le Z_{0}<\dots<Z_{\ell-1}$, one of $Z_{i}$'s is $m$-dense($\Gamma$).
\end{itemize}
Note that there exists a $\Delta_{0}$-formula $\theta(m,Z)$ saying that ``$Z$ is $m$-dense($\Gamma$).''
(Here, we always assume $\min Z>1$ to avoid technical annoyances of the second condition.)
%For the usual Ramsey type statements we consider ($\RT^{n}_{k}$, $\psRT^{2}_{k}$, $\EM$, \dots),
%we usually identify a function $f:[[0,|Z|)_{\N}]^{n}\to k$ with $f\circ (p_{Z}^{-1})^{n}:[Z]^{n}\to k$ and just discuss on $[Z]^{n}$.
\end{defi}
In case $\Gamma$ is $\psRT^{2}_{k}$ or $\RT^{n}_{k}$ for some $n,k\ge 2$, the second condition is implied from the first condition as follows: 
for a given partition $Z_{0}\sqcup\dots \sqcup Z_{\ell-1}=Z$, set $f(x,y)=1$ if $x,y\in Z_{i}$ for some $i<\ell$ and $f(x,y)=0$ otherwise, then, $f$ is a transitive coloring and any $\bbomega$-large homogeneous set $H\subseteq Z$ is included in some $Z_{i}$'s.
 (For more precise explanations, see \cite{Paris1978} or \cite{BW}.)
On the other hand, the density notion for $\EM$ without the second condition does not work well (see \cite{BW}).

\begin{defi}[Paris-Harrington principle for density]
Let $\Gamma$ be a Ramsey-like-$\Pi^{1}_{2}$-statement.
Then, the statement $m$-$\PHt(\Gamma)$ asserts that for any $X_{0}\subseteq\N$, 
	if $X_{0}$ is infinite then there exists an $m$-dense($\Gamma$) set $X$ such that $X\subseteq_{\fin}X_{0}$.
\end{defi}
%Note that $\RT^{n}_{k}$, $\ADS$, $\EM$ are all this form.

%The following lemma is essentially due to Paris\cite{Paris1978}.
The density notion for $\Gamma$ provides a cut to be a model of $\WKLo+\Gamma$.

\begin{lem}\label{lem-cut-from-dense-set}
Let $\Gamma$ be a Ramsey-like-$\Pi^{1}_{2}$-statement.
Given a countable nonstandard model $M$ of $\Ii$ and an $M$-finite set $Z\subseteq M$ which is $a$-dense($\Gamma$) for some $a\in M\setminus \omega$, then there exists an initial segment $I$ of $M$ such that $(I,\Cod(M/I))\models \WKLo+\Gamma$ and $I\cap Z$ is infinite in $I$.
\end{lem}
\begin{proof}
Let $\Gamma$ be a Ramsey-like-$\Pi^{1}_{2}$-statement of the form
$$
	(\forall f:[\N]^{n}\to k)(\exists Y)(Y\mbox{ is infinite}\wedge \Psi(f,Y))
$$
where~$n,k\in\omega$ and~$\Psi$ is of the form in Definition~\ref{def-R-like-and-rest}.
Let $M\models\Ii$ be a countable nonstandard model, and $Z\subseteq M$ be $M$-finite set which is $a$-dense($\Gamma$) for some $a\in M\setminus \omega$.
Let $\{E_{i}\}_{i\in\omega}$ be an enumeration of all $M$-finite sets 
such that any $M$-finite set appears infinitely many times, and
$\{f_{i}\}_{i\in\omega}$ be an enumeration of all $M$-finite functions from $[[0,|Z|)_{\N}]^{n}$ 
to $k$.

In the following, we will construct an $\omega$-length sequence of $M$-finite sets 
$Z=Z_{0}\supseteq Z_{1}\supseteq\dots$ so that for each~$i \in \omega$, $Z_{i}$ is $(a-i)$-dense($\Gamma$), 
$\min Z_{i}<\min Z_{i+3}$, $\Psi(f_{i}\rest[[0,|Z_{3i})_{\N}]^{n}, p_{Z_{3i}}(Z_{3i+1}))$, and $(\min Z_{3i+2},\max Z_{3i+2})_{\N}\cap E_{i}=\emptyset$ if $|E_{i}|<\min Z_{3i+1}$.

At the stage $s=3i$, let $Z_{3i}$ and $f_{i}$ be given.
Then, one can find $Z_{3i+1}\subseteq Z_{3i}$ which is $(a-3i-1)$-dense($\Gamma$) such that $\Psi(f_{i}\rest[[0,|Z_{3i}|)_{\N}]^{n}, p_{Z_{3i}}(Z_{3i+1}))$ by the definition of density($\Gamma$).

At the stage $s=3i+1$, let $Z_{3i+1}$ and $E_{i}$ be given.
If $\min Z_{3i+1}\le |E_{i}|$, let $Z_{3i+2}=Z_{3i+1}.$
If $\min Z_{3i+1}>|E_{i}|$, let $E_{i}=\{e_{0},\dots,e_{l-1}\}$ 
where $e_{0}<e_{1}<\dots<e_{l-1}$, and put 
$W^{0}=Z_{3i+1}\cap [0,e_{0})_{\N}$, $W^{j}=Z_{3i+1}\cap[e_{j-1},e_{j})_{\N}$ 
for $1\le j<l$, and $W^{l}=Z_{3i+1}\cap[e_{l-1},\infty)_{\N}$.
Then, $Z_{3i+1}=W^{0}\sqcup\dots\sqcup W^{l}$, thus one of $W_{j}$'s is $(a-3i-2)$-dense($\Gamma$).
Put $Z_{3i+2}$ to be such $W_{j}$.

At the stage $s=3i+2$, Put $Z_{3i+3}=Z_{3i+2}\setminus\{\min Z_{3i+2}\}$.

Now, let $I=\sup\{\min Z_{i}\mid i\in\omega\}\subseteq_{e} M$.
By the construction of the steps $s=3i+1$, $I$ is a semi-regular cut, thus $(I,\Cod(M/I))\models\WKLo$.
By the construction of the steps $s=3i+2$, $Z_{i}\cap I$ is infinite in $I$ for any $i\in\omega$.
To check that $(I,\Cod(M/I))\models\Gamma$, let $f:[I]^{n}\to k\in \Cod(M/I)$.
Then, there exists $f_{i}$ such that $f=f_{i}\cap I$.
By the construction, $M\models \Psi(f_{i}\rest[[0,|Z_{3i}|)_{\N}]^{n}, p_{Z_{3i}}(Z_{3i+1}))$ holds, thus we have $(I,\Cod(M/I))\models\Psi(f, p_{Z_{3i}}(Z_{3i+1})\cap I)$. Moreover, since $Z_{3i+1}\cap I$ is infinite in $I$ and $p_{Z_{3i}}(x)\le x$ for any $x\in Z$, $p_{Z_{3i}}(Z_{3i+1})\cap I$ is also infinite in $I$.
\end{proof}

Now the density version of Paris-Harrington principle characterize the $\tilde{\Pi}{}^{0}_{3}$-part of Ramsey-like statements.
The following theorem is a generalization of \cite[Theorem~1]{BW}
\begin{thm}\label{thm-conservation-via-density-general}
Let $\Gamma$ be a Ramsey-like-$\Pi^{1}_{2}$-statement.
Then, $\WKLo+\Gamma$ is a $\tilde{\Pi}{}^{0}_{3}$-conservative extension of $\RCAo+\{m\mbox{-}\PHt(\Gamma)\mid m\in\omega\}$.
\end{thm}
\begin{proof}%[\it Proof of Theorem~\ref{thm-conservation-via-density-general}.]
By the usual compactness argument, one can easily check that $\WKLo+\Gamma$ implies $m\mbox{-}\PHt(\Gamma)$ for any $m\in\omega$.
Thus, $\WKLo+\Gamma$ is an extension of $\RCAo+\{m\mbox{-}\PHt(\Gamma)\mid m\in\omega\}$.

To see that it is a $\tilde{\Pi}^{0}_{3}$-conservative extension, let $\varphi_{0}\equiv \A X \A x \E y \A z\varphi(X[z], x,y,z)$ be a $\tilde{\Pi}^{0}_{3}$-sentence which is not provable in $\RCAo+\{m\mbox{-}\PHt(\Gamma)\mid m\in\omega\}$, where $\varphi$ is $\Sigma^{0}_{0}$.
Take a countable nonstandard model $(M,S)\models\RCAo+\{m\mbox{-}\PHt(\Gamma)\mid m\in\omega\}+\neg\varphi_{0}$.
Then, there exist $A\in S$ and $a\in M$ such that $(M,S)\models\A y \E z \neg\varphi(A[z],a,y,z)$.
In $(M,S)$, define a sequence $\langle x_{i}\mid i\in M \rangle$ so that $x_{0}=a$ and $x_{i+1}=\min\{x>x_{i}\mid \A y<x_{i}\E z<x\neg\varphi(A[z],a,y,z)\}$.
By recursive comprehension in $(M,S)$, put $X=\{x_{i}\mid i\in M\}\in S$.
Then, $X$ is infinite in $(M,S)$.
By $m\mbox{-}\PHt(\Gamma)$ for $m\in\omega$, there exist $m$-dense$(\Gamma)$ finite subsets of $X$ for any $m\in\omega$.
Thus, by overspill for $\Sigma^{0}_{1}$-statement, there exists an $m$-dense$(\Gamma)$ finite subset $Z$ of $X$ for some $m\in M\setminus \omega$. 
Now, by Lemma~\ref{lem-cut-from-dense-set}, there exists $I\subsetneq_{e} M$ such that $(I,\Cod(M/I))\models \WKLo+\Gamma$ and $Z\cap I$ is infinite in $I$.
Note that $\Cod(M/I)=\{W\cap I\mid W\mbox{ is $M$-finite}\}=\{W\cap I\mid W\in S\}$.
Since $Z\subseteq X$, $a\le \min Z\in I$ and for any $w,w'\in Z\cap I$ such that $w<w'$, $(I,\Cod(M/I))\models \A y<w\E z<w'\neg\varphi(A\cap I[z],a,y,z)$.
Since $Z\cap I$ is unbounded in $I$, we have $(I,\Cod(M/I))\models \A y\E z\neg\varphi(A\cap I[z],a,y,z)$, which means $(I,\Cod(M/I))\models\neg\varphi_{0}$.
Thus, $\WKLo+\Gamma$ does not prove $\varphi_{0}$.
\end{proof}

The density notion actually captures some finite consequences of Ramsey-like-statements as follows.

%\begin{thm}[Generalized Parsons theorem]\label{thm-g-Parsons}
%Let $\psi(F)$ be a $\Delta_{0}$-formula. Assume that
%\begin{itemize}
% \item[] \rm $\II\vdash \A X\subseteq \N(X$ is infinite $\to \E F\subseteq_{\fin} X\psi(F))$.
%\end{itemize}
%Then, there exists $n\in\omega$ such that
%\begin{itemize}
% \item[] \rm $\Ii\vdash \A Z\subseteq_{\fin}(0,\infty)_{\N}(Z$ is $\bbomega^{n}$-large $\to \E F\subseteq Z\psi(F))$.
%\end{itemize}
%\end{thm}
\begin{thm}\label{thm-g-Parsons}
Let $\Gamma$ be a Ramsey-like-$\Pi^{1}_{2}$-statement,
and let $\psi(x,y,F)$ be a $\Delta_{0}$-formula with exactly the displayed free variables. Assume that
\begin{itemize}
 \item[] \rm $\WKLo+\Gamma\vdash \A x \A X(X$ is infinite $\to \E F\subseteq_{\fin} X\E y\psi(x,y,F))$.
\end{itemize}
Then, there exists $n\in\omega$ such that
\begin{itemize}
 \item[] \rm $\Ii\vdash \A x \A Z\subseteq_{\fin}(x,\infty)_{\N}(Z$ is $n$-dense$(\Gamma)\to \E F\subseteq Z\E y<\max Z\psi(x,y,F))$.
\end{itemize}
\end{thm}
\begin{proof}%[Proof of Theorem~\ref{thm-g-Parsons}.]
%This proof is a generalization of the Friedman's proof of Theorem~\ref{thm-Friedman-conservation}, which is the Parsons theorem extended into $\WKLo$.
Assume that $\II\not\vdash \A x \A Z\subseteq_{\fin}(x,\infty)_{\N}(Z$ is $n$-dense$(\Gamma)\to \E F\subseteq Z\E y<\max Z\psi(x,y,F))$
for any $n\in\omega$.
Then, there exists a countable model $M\models\Ii+\{\E x\E Z\subseteq_{\fin}(x,\infty)_{\N}(Z$ 
is $n$-dense$(\Gamma)\wedge \A F\subseteq Z\A y<\max Z\neg\psi(x,y,F))\mid n\in\omega\}$ such that $M\not\cong\omega$.
By overspill, there exists $a\in M\setminus \omega$ such that $M\models \E x\E Z\subseteq_{\fin}(x,\infty)_{\N}(Z$ 
is $a$-dense$(\Gamma)\wedge \A F\subseteq Z\A y<\max Z\neg\psi(x,y,F))$,
thus there exist $c\in M$ and an $M$-finite set 
$Z\subseteq M$ with $\min Z>c$ such that $Z$ is $a$-dense$(\Gamma)$ and $\A F\subseteq Z\A y<\max Z\neg\psi(c,y,F)$.

Now, by Lemma~\ref{lem-cut-from-dense-set}, there exists $I\subsetneq_{e} M$ such that $(I,\Cod(M/I))\models \WKLo+\Gamma$ and $Z\cap I$ is infinite in $I$.
Note that $c\in I$.
Thus, we have $(I,\Cod(M/I))\models (Z\cap I$ is infinite $\wedge \A F\subseteq_{\fin} Z\cap I\A y\neg\psi(c,y,F)))$.
This contradicts to $\WKLo+\Gamma\vdash \A x \A X(X$ is infinite $\to \E F\subseteq_{\fin} X\E y\psi(x,y,F))$.
\end{proof}
The argument we used in Lemma~\ref{lem-cut-from-dense-set}, Theorems~\ref{thm-conservation-via-density-general} and \ref{thm-g-Parsons} is a generalization of (a special case of) the well-known indicator argument (see, e.g., \cite{Paris1978,Kaye}).
Actually, by Theorem~\ref{thm-conservation-via-density-general} and Proposition~\ref{prop-R-like-and-rest-formulas}, 
one can characterize the $\tilde{\Pi}^{0}_{3}$-part of any restricted-$\Pi^{1}_{2}$-statement, as same as the usual indicator arguments captures $\Pi^{0}_{2}$-parts.
In general, one can replace the second condition and the initial condition for $0$-density in Definition~\ref{def-density-with-indicator} with suitable indicator conditions for a base system $T$, then the partial conservation for $T+\Gamma$ over $T+\{m\mbox{-}\PHt(\Gamma)\mid m\in\omega\}$ holds, and one can even consider the $\tilde{\Pi}^{0}_{4}$-part in some cases.
%See \cite{Y-itrt,Y-RIMS2013,Y-MLQ13}.
See \cite{Y-RIMS2013,Y-MLQ13}.

The following corollary of the previous theorem plays a key role in this paper.
\begin{cor}[Generalized Parsons theorem]\label{cor-g-Parsons}
Let $\psi(F)$ be a $\Sigma_{1}$-formula with exactly the displayed free variables. Assume that
\begin{itemize}
 \item[] \rm $\II\vdash \A X\subseteq \N(X$ is infinite $\to \E F\subseteq_{\fin} X\psi(F))$.
\end{itemize}
Then, there exists $n\in\omega$ such that
\begin{itemize}
 \item[] \rm $\Ii\vdash \A Z\subseteq_{\fin}(0,\infty)_{\N}(Z$ is $\bbomega^{n}$-large $\to \E F\subseteq Z\psi(F))$.
\end{itemize} 
\end{cor}
\begin{proof}
By Lemma~\ref{lem-alpha-small-sum}, any $\bbomega^{n+1}$-large set is $n$-dense$(0=0)$ (dense for the trivial statement).
Thus, we have this corollary as a special case of Theorem~\ref{thm-g-Parsons}.
\end{proof}
Note that this corollary quickly implies (a weaker version of) the Parsons theorem (see, e.g., \cite{handbook-of-proof}), namely, any $\Pi_{2}$-statement provably in $\Ii$ is bounded by a primitive recursive function, as follows.
If a $\Pi_{2}$-statement $\A x\E y\theta(x,y)$ is provable within $\II$, then put $\psi(F):\equiv (\A x<\min F) (\E y<\max F)\theta(x,y)$.
Then, $\II$ proves that any infinite set contains a finite set $F$ such that $\psi(F)$ holds.
By this theorem, there exists $n\in\omega$ such that $(Z$ is $\bbomega^{n}$-large $\to \E F\subseteq Z\psi(F))$.
One can easily find a primitive recursive function $h$ such that $[a,h(a)]_{\N}$ is $\bbomega^{n}$-large for any $a\in\omega$.
Thus, we have $\Ii\vdash \A x \E y<h(x)\theta(x,y)$.
Note also that one can apply the generalized Parsons theorem for any 
$\tilde{\Pi}{}^{0}_{3}$-conservative extension of $\II$, e.g., $\WKLo+\BII$.

We are now ready to prove the main conservation theorem of the section.

\begin{proof}[Proof of Theorem~\ref{thm-provably-large-conservative}]
Let $\Gamma$ be a Ramsey-like-$\Pi^{1}_{2}$-statement, and assume that
for any $k\in\omega$,
\begin{itemize}
 \item[] \rm $\II\vdash \A X\subseteq \N(X$ is infinite $\to \E F\subseteq_{\fin} X(F$ is $\bbomega^{k}$-large$(\Gamma)))$.
\end{itemize}
Then, by Corollary~\ref{cor-g-Parsons}, for each $k\in\omega$ there exists $n_{k}$ such that
\begin{itemize}
 \item[] \rm $\II\vdash \A Z\subseteq_{\fin}(0,\infty)_{\N}(Z$ is $\bbomega^{n_{k}}$-large $\to Z$ is $\bbomega^{k}$-large$(\Gamma))$.
\end{itemize}
Now, put $h:\omega\to\omega$ as $h(0)=1$ and $h(m+1)=\max\{n_{h(m)}, h(m)+1\}$.
We will check
\begin{itemize}
 \item[] \rm $\II\vdash \A Z\subseteq_{\fin}(0,\infty)_{\N}(Z$ is $\bbomega^{h(m)}$-large $\to Z$ is $m$-dense$(\Gamma))$.
\end{itemize}
by induction.
The case $m=0$ follows from the definition.
The case $m=m'+1$, $\bbomega^{h(m')}$-large sets are $m'$-dense$(\Gamma)$ by the induction hypothesis.
Then, the first condition of the $m'+1$-density follows from $h(m'+1)\ge n_{h(m')}$,
and the second condition follows from $h(m'+1)\ge h(m')+1$ and Lemma~\ref{lem-alpha-small-sum}.
Thus, by Theorem~\ref{thm-omega-largeness}, $\II$ proves that any infinite set contains an $m$-dense$(\Gamma)$ set for any $m\in\omega$.
Hence, by Theorem~\ref{thm-conservation-via-density-general}, $\WKLo+\Gamma$ is a $\tilde{\Pi}{}^{0}_{3}$-conservative extension of $\II$.
\end{proof}

When two conservation results are obtained, one can often amalgamate those results.
For example, if two $\Pi^{1}_{2}$-theories $T_{1}$ and $T_{2}$ are $\Pi^{1}_{1}$-conservative over a base $\Pi^{1}_{2}$-theory $T_{0}$, then $T_{1}+T_{2}$ is also $\Pi^{1}_{1}$-conservative over $T_{0}$ (see \cite{Y2009}).
Similar amalgamation property holds for $\tilde\Pi^{0}_{3}$-conservation as follows.
\begin{thm}[Amalgamation]\label{thm:amalgamation}
Fix $n\ge 1$.
Let $T$ be a theory extending $\II$ which consists of sentences of the form $\A X \E Y \theta(X,Y)$ where $\theta$ is $\Pi^{0}_{n+2}$,
and let $\Gamma_{1}$ and $\Gamma_{2}$ be sentences of the same form as $T$.
If $T+\Gamma_{i}$ is a $\tilde\Pi^{0}_{n+2}$-conservative extension of $T$ for $i=1,2$,
then, $T+\Gamma_{1}+\Gamma_{2}$ is a $\tilde\Pi^{0}_{n+2}$-conservative extension of $T$.
\end{thm}
\begin{proof}
The proof is essentially the same as the case for the amalgamation of two $\Pi^{1}_{2}$-theories which are $\Pi^{1}_{1}$-conservative over a base theory in \cite{Y2009}.
Here, we consider the $\mathcal{L}_{2}$-structure $(M,S)$ as a two-sorted structure, namely, $M$ and $S$ are disjoint sets and $\in$ is a relation on $M\times S$.
In this understanding, $\tilde\Pi^{0}_{n+2}$-conservation implies that any model of $T$ has a $\Sigma^{0}_{n+2}$-elementary extension which is a model of $T+\Gamma_{i}$ for $i=1,2$.
(This is because if $(M,S)\models T$, then $\Th_{\mathcal{L}_{2}\cup M \cup S}(M,S)\cap \Sigma^{0}_{n+2}+T+\Gamma_{i}$ is consistent.)
Now, assume $T\not\vdash\A X\psi(X)$ where $\psi$ is $\Pi^{0}_{n+2}$, and take a model $(M_{0},S_{0})\models T+\E X\neg\psi(X)$.
Then, one can construct a $\Sigma^{0}_{n+2}$-elementary chain of models $(M_{0},S_{0})\subseteq(M_{1},S_{1})\subseteq\dots$ such that $(M_{2j+i},S_{2j+i})\models T+\Gamma_{i}$ for $i=1,2$ and $j\in \omega$.
By the usual elementary chain argument, $(\bar{M},\bar{S})=(\bigcup_{k\in\omega}M_{k}, \bigcup_{k\in\omega} S_{k})$ is a $\Sigma^{0}_{n+2}$-elementary extension of $(M_{0},S_{0})$, and therefore $(\bar{M},\bar{S})\models T+\Gamma_{1}+\Gamma_{2}+\E X\neg\psi(X)$.
Hence $T+\Gamma_{1}+\Gamma_{2}\not\vdash \A X\psi(X)$.
\end{proof}
Note that $\II$, $\WKLo$ and any Ramsey-like statement is of the form $\A X \E Y \theta(X,Y)$ where $\theta$ is $\Pi^{0}_{3}$.
Therefore, one can always use the amalgamation theorem for $\tilde\Pi^{0}_{3}$-conservation.
In particular, to prove the $\tilde\Pi^{0}_{3}$-conservation theorem for $\RT^{2}_{2}$, we only need to prove $\tilde\Pi^{0}_{3}$-conservation theorems for $\ADS$ and $\EM$.

\section{Conservation theorem for $\ADS$}\label{section-ADS}

In this section, we will show that $\WKLo+\ADS$ is a
 $\tilde{\Pi}{}^{0}_{3}$-conservative extension of $\II$.
Actually, this is just a weakening of the following theorem by Chong, Slaman and Yang,
where $\CAC$ is the chain antichain principle, since $\ADS$ is a consequence of $\CAC$
over~$\RCAo$~\cite{HS2007}.

\begin{thm}[Chong, Slaman, Yang~\cite{CSY2012}]\label{thm-CSY-conservation}
$\WKLo+\CAC$ is a ${\Pi}{}^{1}_{1}$-conservative extension of $\BII$. 
\end{thm}

Here, we will give an alternative proof by calculating the size 
of $\bbomega^{k}$-large($\psRT^2_2$) sets.
To simplify the proof below, we will use a slightly modified $\alpha$-largeness notion.

\begin{defi}[$\II$]
Any set is said to be $0$-large$^{*}$. Given some $\alpha<\bbomega^{\bbomega}$, 
$X\subseteq_{\fin}\N$ is said to be $\alpha$-large$^{*}$ if
\begin{itemize}
 \item $X\setminus \{\min X\}$ is $\beta$-large$^{*}$ if $\alpha=\beta+1$,
 \item $X$ is $(\beta+\bbomega^{n-1}\cdot\min X)$-large$^{*}$ if $\alpha=\beta+\bbomega^{n}$.
\end{itemize}
\end{defi}

Trivially, if $X\subseteq_{\fin}\N$ is $\alpha$-large, then $X$ is $\alpha$-large$^{*}$.

\begin{lem}
For any $k\in\omega$, the following is provable within $\II$.
For any $\alpha<\bbomega^{k}$ and for any $X\subseteq_{\fin}\N$, 
$X$ is $\alpha$-large if $X$ is $\alpha+1$-large$^{*}$.
\end{lem}
\begin{proof}
By $\Pi^{0}_{1}$-transfinite induction up to $\bbomega^{k}$, which is available within $\RCAo$.
\end{proof}

\begin{lem}\label{lem-alpha*-large-sum}
The following is provable within $\II$.
For any $k,n\in\N$, if $X$ is a disjoint union of $X_{0},\dots,X_{k-1}$ 
such that $X_{i}<X_{i+1}$ for any $i<k-1$ and each $X_{i}$ 
is $\bbomega^{n}$-large$^{*}$, then $X$ is $\bbomega^{n}\cdot k$-large$^{*}$. 
Thus, if $k\le \min X_{0}$, $X$ is $\bbomega^{n+1}$-large$^{*}$.
\end{lem}
\begin{proof}
Similar to Lemma~\ref{lem-alpha-large-sum}.
\end{proof}

\begin{lem}\label{lem-ADS-largeness}
For any $k\in\omega$, the following is provable within $\II$. 
If a finite set $X\subseteq\N$ is $\bbomega^{2k+6}$-large 
and $\min X>3$, then $X$ is $\bbomega^{k}$-large($\psRT^{2}_{2}$).
\end{lem}
\begin{proof}
%Given a coloring $f:[[0,\max X]_{\N}]^{2}\to 2$,
Given a coloring $f:[X]^{2}\to 2$,
 define the coloring 
$\bar{f}:[X]^{2}\to 2k+2$ as $\bar{f}(x,y)=2j+i$  if $f(x,y)=i$ 
and $j=\min\{j'<k\mid \neg(\E H\subseteq[x,y)_{\N}\cap X$ $x\in H$, 
$H$ is $\bbomega^{j'+1}$-large$^{*}$ and $H\cup\{y\}$ is pseudo-homogeneous for~$f$ with the color $i)\}\cup\{k\}$.

By Theorem~\ref{SK-theorem}, take $Y\subseteq X$ such that 
$Y$ is $\bbomega$-large and $\bar{f}$-homogeneous.
Let $Y=\{y_{0}<y_{1}<\dots<y_{l}\}$, and $\bar{f}([Y]^{2})=2j+i$.
Then, $l\ge y_{0}$. By the definition of $\bar{f}$, for $s=0,\dots,l-1$, 
one can take $H_{s}\subseteq[y_{s},y_{s+1})_{\N}$ such that $y_{s}\in H_{s}$, 
$H_{s}$ is $\bbomega^{j}$-large$^{*}$ and $[H_{s}\cup\{y_{s+1}\}]^{2}$ is pseudo-homogeneous for~$f$ with the color $i$.
By Lemma~\ref{lem-alpha*-large-sum}, $H=\bigcup_{s=0}^{l-1}H_{s}$ is 
$\bbomega^{j+1}$-large$^{*}$, and $[H\cup\{y_{k}\}]^{2}$ is pseudo-homogeneous for~$f$ with the color $i$.
This $H$ assures that $\bar{f}(y_{0},y_{k})\neq 2j+i$ or $j=k$. Thus, we have $j=k$.
Hence $H$ is pseudo-homogeneous for~$f$ and $\bbomega^{k+1}$-large$^{*}$, thus it is $\bbomega^{k}$-large.
\end{proof}

\begin{cor}\label{cor-con-ADS}
$\WKLo+\psRT^{2}_{2}$, or equivalently, $\WKLo+\ADS$ is a $\tilde{\Pi}^0_3$-conservative extension of $\II$. 
\end{cor}
\begin{proof}
By Theorems~\ref{thm-omega-largeness}, 
\ref{thm-provably-large-conservative} and Lemma~\ref{lem-ADS-largeness}.
\end{proof}

Note that we could have proven Corollary~\ref{cor-con-ADS} by working with $\ADS$ directly.
However, the unnatural formulation of $\ADS$ as a Ramsey-like-$\Pi^1_2$-statement
introduces additional technicalities in the proof. Indeed, the standard formulation of $\ADS$
involves linear orders, whereas a Ramsey-like statement is about arbitrary coloring functions.
In this framework, a solution to $\ADS$ is either an infinite homogeneous set, or a set whose minimal element witnesses 
the non-transitivity of the coloring.

\section{Grouping principle}\label{section-GP}

In this section, we introduce the grouping principle, which is a consequence
of Ramsey's theorem. The grouping principle will be used in the conservation proof of the
Erd\H{o}s-Moser theorem, although it is currently unknown how the two statements relate in reverse mathematics.
The grouping principle seems interesting to study in its own right, and we conduct a study
of its relations with other Ramsey-type principles already studied in reverse mathematics.

\begin{definition}[$\RCAo$, grouping principle]
Given a largeness notion $\LL$ and a coloring $f:[\N]^{n}\to k$, 
an \textit{$\LL$-grouping for $f$} is an infinite family 
of $\LL$-large finite sets $\{F_0 < F_1 < \dots \}\subseteq \LL$ such that 
$$
\A i_{1}<\dots< i_{n}\, \E c<k\, \A x_{1}\in F_{i_{1}},\dots,\A x_{n}\in F_{i_{n}}\, f(x_{1},\dots,x_{n})=c
$$
Now $\GPg^{n}_{k}(\LL)$ (grouping principle for $\LL$) asserts 
that for any coloring $f:[\N]^{n}\to k$, there exists an infinite $\LL$-grouping for $f$.
We write $\GPg^{n}_{k}$ for the statement saying that for any largeness notion 
$\LL$, $\GPg^{n}_{k}(\LL)$ holds, $\GPg^{n}$ for $\A k\GPg^{n}_{k}$, 
and $\GPg$ for $\A n\GPg^{n}$.
\end{definition}

Note that being a largeness notion is a $\Pi^1_1$-statement. Therefore, an \emph{instance}
of~$\GP$ is a pair~$\langle \LL, f \rangle$ where $\LL$ is a collection of finite sets,
and $f : [\N]^2 \to 2$ is a coloring.  A \emph{solution} to an instance $\langle \LL, f \rangle$
is either an $\LL$-grouping for~$f$, or an infinite set witnessing that $\LL$ is not a largeness notion,
that is, an infinite set with no finite subset in~$\LL$.

In order to simplify the analysis of Ramsey's theorem for pairs,
Cholak, Jockusch and Slaman~\cite{CJS} split the proof of~$\RT^2_2$ into \emph{cohesiveness}
and a stable restriction of $\RT^2_2$. A coloring $f : [\N]^2 \to k$ is \emph{stable}
if for every~$x$, $\lim_y f(x, y)$ exists. $\SRT^2_k$ is the restriction of $\RT^2_2$ to stable colorings.

\begin{defi}[Cohesiveness]
An infinite set $C$ is $\vec{R}$-cohesive for a sequence of sets $R_0, R_1, \dots$
if for each $i \in \N$, $C \subseteq^{*} R_i$ or $C \subseteq^{*} \overline{R_i}$.
$\COH$ is the statement ``Every uniform sequence of sets $\vec{R}$
has an $\vec{R}$-cohesive set.''
\end{defi}

Cohesiveness is a statement from standard computability theory.
Cholak, Jockusch and Slaman~\cite{CJS} claimed with an erroneous proof that it is a strict consequence of $\RT^2_2$
over~$\RCAo$. Mileti~\cite{Mileti2004Partition} fixed the proof. Hirschfeldt and Shore~\cite{HS2007}
proved that $\COH$ is a consequence of~$\ADS$. Since then, many statements in reverse mathematics
have been split into their cohesive and their stable part~\cite{HS2007}.
Accordingly, we will consider  the stable version $\SGP$ which stands for $\GP$ for stable colorings.
One can prove that~$\RCAo \vdash \COH + \SGP \rightarrow \GP$ by the same argument
as $\RCA \vdash \COH + \SRT^2_2 \rightarrow \RT^2_2$ ~\cite{CJS}. 
Stable Ramsey's theorem for pairs admits a nice computability-theoretic characterization
in terms of infinite subsets of a $\Delta^0_2$ set.
We can give a similar characterization for the stable grouping principle for pairs.

\begin{definition}[$\RCAo$]
Given a largeness notion~$\LL$, an \textit{$\LL$-grouping} for a set~$A$ 
is an infinite family of $\LL$-large finite sets
$\{F_0 < F_1 < \dots \} \subseteq \LL$ such that
$(\forall i)[F_i \subseteq A \vee F_i \subseteq \overline{A}]$
\end{definition}

The argument can be carried-out within $\BII$. Actually, Kreuzer proved that $\SGP$ implies~$\BII$ over~$\RCAo$ (see Theorem~\ref{thm:sgp-bii}).

\begin{lemma}\label{lem:stable-grouping-characterization}
$\RCAo + \BII \vdash \SGP(\LL) \leftrightarrow$ ``Every $\Delta^0_2$ set has an infinite~$\LL$-grouping''.
\end{lemma}

We will now show the existence of an $\omega$-model of $\SGP$ containing only low sets.
Recall that an instance is a pair $\langle \LL, f \rangle$,
and a solution is either a witness that $\LL$ is not a largeness notion, or an $\LL$-grouping for~$f$.
We need therefore to show that given any computable collection of finite sets~$\LL$
and any $\Delta^0_2$ set~$A$, there is either an infinite low set~$Y$ witnessing that
$\LL$ is not a largeness notion, or a low $\LL$-grouping for~$A$.
In what follows, we denote by $\mathsf{LOW}$ the collection of all low sets.

\begin{theorem}\label{thm:grouping-low-solution}
For every computable set~$\LL$ which is a largeness notion on $(\omega, \mathsf{LOW})$,
every $\Delta^0_2$ set has an infinite low $\LL$-grouping.
\end{theorem}
\begin{proof}
Fix a $\Delta^0_2$ set~$A$.
We will construct an infinite low $\LL$-grouping for~$A$ by an effective forcing notion
whose conditions are tuples~$c = (F_0, \dots, F_k, X_0, \dots, X_m)$ such that
\begin{itemize}
	\item[(i)] $F_i \in \LL$ and $F_i \subseteq_{\fin} A$ or~$F_i \subseteq_{\fin} \overline{A}$ for each~$i \leq k$,
	\item[(ii)] $F_i < F_{i+1}$ for each~$i < k$,
	\item[(iii)] $X_0 \sqcup \dots \sqcup X_m$ is a low partition of~$\omega$.
\end{itemize}
A condition~$d = (F_0, \dots, F_\ell, Y_0, \dots, Y_n)$ \emph{extends} a condition~$c = (F_0, \dots, F_k, X_0, \dots, X_m)$
(written~$d \leq c$) if~$\ell \geq k$, for every~$i \in (k, \ell]_{\N}$, $F_i \subseteq X_j$ for some~$j \leq m$ 
and~$Y_0, \dots, Y_n$ refines $X_0, \dots, X_m$, that is, for each~$i \leq n$, there is some~$j \leq m$ such that~$Y_i \subseteq X_j$.
An \emph{index} of the condition~$c$ is a tuple~$(F_0, \dots, F_k, e)$ where~$\Phi^{\emptyset'}_e$ decides the jump
of the partition~$X_0, \dots, X_m$.
We first claim that the finite sequence of sets can be extended.

\begin{claim*}
For every condition $c = (F_0, \dots, F_k, X_0, \dots, X_m)$,
there is an extension~$d = (F_0, \dots, F_\ell,\allowbreak X_0, \dots, X_m)$ of~$c$ such that~$\ell > k$.
Moreover, an index of~$d$ can be found $\emptyset'$-uniformly in an index of~$c$.
\end{claim*}
\begin{proof}[\it Proof of the claim]
We first show that there is a set~$F > F_k$ such that~$F \in \LL$ and $F \subseteq X_i \cap A$ or $F \subseteq X_i \cap \overline{A}$ for some~$i \leq n$.
Let~$i \leq n$ be such that~$X_i$ is infinite.
We claim that there is some finite set~$F \in \LL$ such that $F \subseteq X_i \cap A \setminus [0, \max(F_k)]_{\N}$
or $F \subseteq X_i \cap \overline{A} \setminus [0, \max(F_k)]_{\N}$.
Suppose for the sake of contradiction that there is no such set. Then the $\Pi^{0,X_i}_1$ class
of all sets~$Z$ such that for every $F \in \LL$,  $F \not\subseteq X_i \cap Z \setminus [0, \max(F_k)]_{\N}$
and $F \not \subseteq X_i \cap \overline{Z} \setminus [0, \max(F_k)]_{\N}$ is non-empty.
By the low basis theorem, there is a $Z$ such that~$Z \oplus X_i$ is low over~$X_i$,
hence low. The set~$Z$ or its complement contradicts the fact that 
$\LL$ is a largeness notion on $(\omega, \mathsf{LOW})$.

Knowing that such a set~$F$ exists, we can find it $\emptyset'$-uniformly in $c$ and a $\Delta^0_2$ index of the set~$A$.
The condition~$(F_0, \dots, F_k, F, X_0, \dots, X_m)$ is a valid extension of~$c$.
Note that such a set~$F$ does not need to be part of an infinite~$X_i \cap A$ or~$X_i \cap \overline{A}$.
The choice of an infinite part has simply been used to claim the existence of any such set.
\end{proof}

We say that an $\LL$-grouping for~$A$ $\langle E_0 < E_1 < \dots\rangle$ \emph{satisfies} a condition
$c = (F_0, \dots, F_k, X_0, \dots, X_m)$ if $E_0 = F_0$, \dots, $E_k = F_k$
and for every~$i > k$, there is some~$j \leq m$ such that $E_i \subseteq X_j$.
A condition~$c$ \emph{forces} formula $\varphi(G)$ if $\varphi(\vec{E})$
holds for every $\LL$-grouping $\vec{E}$ satisfying~$c$.

\begin{claim*}
For every condition~$c$ and every index~$e \in \omega$, there is an extension~$d$
forcing either~$\Phi^G_e(e) \downarrow$ or~$\Phi^G_e(e) \uparrow$. Moreover, an index of~$d$
can be found $\emptyset'$-uniformly in an index of~$c$ and~$e$.
\end{claim*}
\begin{proof}[\it Proof of the claim]
Fix a condition $c = (F_0, \dots, F_k, X_0, \dots, X_m)$. We have two cases.

In the first case, for every 2-partition~$Z_0 \cup Z_1 = \omega$,
there is a sequence of finite sets~$F_{k+1}, \dots, F_\ell$ such that
$F_k < F_{k+1} < \dots < F_\ell$, $\Phi_e^{F_0, \dots, F_\ell}(e) \downarrow$,
and for every~$i \in (k, \ell]$, $F_i \in \LL$ and there is some~$j \leq m$
$F_i \subseteq Z_0 \cap X_j$ or~$F_i \subseteq Z_1 \cap X_j$.
In particular, taking~$Z_0 = A$ and~$Z_1 = \overline{A}$, 
there is a sequence of finite sets~$F_{k+1}, \dots, F_\ell$
such that~$d = (F_0, \dots, F_\ell, X_0, \dots, X_m)$ is a valid extension of~$c$
and $\Phi_e^{F_0, \dots, F_\ell}(e) \downarrow$. Such an extension can be found $\emptyset'$-uniformly in
an index of~$c$, $e$ and a $\Delta^0_2$ index of~$A$.

In the second case, the $\Pi^{0, \vec{X}}_1$ class of all the 2-partitions~$Z_0 \cup Z_1 = \omega$
such that $\Phi_e^{F_0, \dots, F_\ell}(e) \uparrow$ for every sequence of finite sets~$F_{k+1}, \dots, F_\ell$ such that
$F_k < F_{k+1} < \dots < F_\ell$,
and for every~$i \in (k, \ell]_{\N}$, $F_i \in \LL$ and there is some~$j \leq m$
$F_i \subseteq Z_0 \cap X_j$ or~$F_i \subseteq Z_1 \cap X_j$ is non-empty.
By the low basis theorem~\cite{Jockusch197201} relativized to~$\vec{X}$, there is a such a 2-partition~$Z_0 \cup Z_1 = \omega$
such that~$Z_0 \oplus Z_1 \oplus \vec{X}$ is low. Moreover, a lowness index for~$Z_0 \oplus Z_1 \oplus \vec{X}$
can be found uniformly in a lowness index for~$\vec{X}$.
The condition~$d = (F_0, \dots, F_k, X_0 \cap Z_0, X_0 \cap Z_1, \dots, X_m \cap Z_0, X_m \cap Z_1)$ 
is an extension of~$c$ forcing~$\Phi^G_e(e) \uparrow$.

Moreover, we can $\vec{X}'$-decide (hence $\emptyset'$-decide) whether the $\Pi^{0, \vec{X}}_1$ class is empty,
thus we can find the extension~$d$ $\emptyset'$-effectively in an index of~$c$ and~$e$.
\end{proof}

Thanks to the claims, define an infinite, 
uniformly $\emptyset'$-computable decreasing sequence of conditions
$c_0 = (\varepsilon, \omega) \geq c_1 \geq c_2 \geq \dots$, 
where~$c_s = (F_0, \dots, F_{k_s}, X^s_{0}, \dots, X^s_{m_s})$ such that for each~$s \in \omega$
\begin{itemize}
	\item[(a)] $k_s \geq s$
	\item[(b)] $c_{s+1}$ forces $\Phi^G_s(s) \uparrow$ or~$\Phi^G_s(s) \downarrow$
\end{itemize}
This sequence yields a $\LL$-grouping for~$A$ $\langle F_0, F_1, \dots \rangle$
which is infinite by (a) and whose jump is $\Delta^0_2$ by (b). 
This finishes the proof of Theorem~\ref{thm:grouping-low-solution}.
\end{proof}

\begin{corollary}\label{cor:sgp-only-low-sets}
$\SGP + \SADS + \WKL$ has an $\omega$-model with only low sets.
\end{corollary}
\begin{proof}
By Lemma~\ref{lem:stable-grouping-characterization} and 
and the low basis theorem~\cite{Jockusch197201} in a relativized form.
As explained, for every collection~$\LL$ and every stable coloring $f : [\omega]^2 \to 2$, one need either to add a
low set witnessing that $\LL$ is not a notion of largeness, or to add an infinite low $\LL$-grouping for~$f$.
\end{proof}

\begin{corollary}
$\RCAo + \SGP + \SADS + \WKL$ implies neither~$\SRT^2_2$, nor~$\SEM$.
\end{corollary}
\begin{proof}
Downey, Hirschfeldt, Lempp and Solomon~\cite{Downey200102}
built a computable instance of $\SRT^2_2$ with no low solution.
Corollary~\ref{cor:sgp-only-low-sets} enables us to conclude that~$\RCAo + \SGP + \SADS + \WKL$
does not imply $\SRT^2_2$. Since~$\RCAo + \SADS + \SEM$ implies~$\SRT^2_2$ (see \cite{Lerman2013Separating})
then $\RCAo + \SGP + \SADS + \WKL$ does not imply $\SEM$.
\end{proof}

Among the computability-theoretic properties used to separate
Ramsey-type theorems in reverse mathematics, the framework
of preservation of hyperimmunity has been especially fruitful.

\begin{definition}[Hyperimmunity]
The \emph{principal function} of a set $B = \{x_0 < x_1 < \dots \}$ 
is the function $p_B$ defined by~$p_B(i) = x_i$ for each~$i \in \N$.
A set~$X$ is \emph{hyperimmune} if its principal function is not dominated by any
computable function. 
\end{definition}

Wang~\cite{Wang2014Definability} recently used the notion of preservation
of the arithmetic hierarchy to separate various theorems in reverse mathematics.
The first author showed~\cite{Patey2015Iterative} that a former separation
of the Erd\H{o}s-Moser theorem from stable Ramsey's theorem for pairs
due to Lerman, Solomon and Towsner~\cite{Lerman2013Separating}
could be reformulated in a similar framework, yielding the notion
of preservation of hyperimmunity.

\begin{defi}[Preservation of hyperimmunity]
A $\Pi^1_2$-statement~$\Psf$ \emph{admits preservation of hyperimmunity} if for each set~$Z$, 
each $Z$-hyperimmune sets~$A_0, A_1, \dots$,
and each $\Psf$-instance $X \leq_T Z$,
there is a solution $Y$ to~$X$ such that the~$A$'s are $Y \oplus Z$-hyperimmune.
\end{defi}

In particular, if a $\Pi^1_2$-statement $\Psf$
admits preservation of hyperimmunity but another statement~$\Qsf$ does not,
then $\Psf$ does not imply~$\Qsf$ over~$\RCAo$.
We now show that the grouping principle enjoys preservation of hyperimmunity
and deduce several separations from it.

\begin{theorem}\label{thm:sgp22-preservation-hyperimmunity}
$\SGP$ admits preservation of hyperimmunity.
\end{theorem}
\begin{proof}
Let~$C$ be a set and~$B_0, B_1, \dots$ be a sequence of $C$-hyperimmune sets.
Let~$S$ be the collection of all sets~$X$ such that the~$B$'s are $X \oplus C$-hyperimmune.
By Lemma~\ref{lem:stable-grouping-characterization}, it suffices to show that for every
$\Delta^{0,C}_2$ set~$A$ and every $C$-computable largeness notion~$\LL$ on $(\omega, S)$
there is an infinite $\LL$-grouping $\vec{F} = \langle F_0 < F_1 < \dots \rangle$ for~$A$
such that the~$B$'s are $\vec{F} \oplus C$-hyperimmune.
Therefore, every instance $\langle \LL, A \rangle$ will have a solution $Y \in S$,
which will be either a witness that $\LL$ is not a largeness notion, or an $\LL$-grouping for~$A$.

Fix~$A$ and~$\LL$.
We will construct an infinite $\LL$-grouping for~$A$ by a forcing argument
whose conditions are tuples~$(F_0, \dots, F_k, X)$ where
\begin{itemize}
	\item[(i)] $F_i \in \LL$ and $F_i \subseteq_{\fin} A$ or~$F_i \subseteq_{\fin} \overline{A}$ for each~$i \leq k$.
	\item[(ii)] $F_i < F_{i+1}$ for each~$i < k$
	\item[(iii)] $X$ is an infinite set such that the~$B$'s are~$X \oplus C$-hyperimmune.
\end{itemize}
A condition~$d = (F_0, \dots, F_\ell, Y)$ \emph{extends} a condition~$c = (F_0, \dots, F_k, X)$
(written~$d \leq c$) if~$\ell \geq k$ and for every~$i \in (k, \ell]_{\N}$, $F_i \subseteq X$.
The proof of the following claim is exactly the same as in Theorem~\ref{thm:grouping-low-solution},
using the hyperimmune-free basis theorem instead of the low basis theorem.

\begin{claim*}
For every condition $c = (F_0, \dots, F_k, X)$,
there is an extension~$d = (F_0, \dots, F_\ell, Y)$ of~$c$ such that~$\ell > k$.
\end{claim*}

The following claim shows that every sufficiently generic filter yields a sequence~$\vec{F}$
such that the~$B$'s are~$\vec{F} \oplus C$-hyperimmune. The notion of satisfaction and of forcing
a formula $\varphi(G)$ are defined as in Theorem~\ref{thm:grouping-low-solution}.

\begin{claim*}
For every condition~$c$ and every pair of indices~$e,i \in \omega$, there is an extension~$d$
forcing $\Phi^{G \oplus C}_e$ not to dominate~$p_{B_i}$.
\end{claim*}
\begin{proof}[\it Proof of the claim]
Fix a condition $c = (F_0, \dots, F_k, X)$.
Let $f$ be the function which on input~$x$, searches 
for a finite set of integers~$U$ such that for every 2-partition~$Z_0 \cup Z_1 = X$,
there is some finite sequence of sets~$F_{k+1}, \dots, F_\ell$ such that
$F_k < \dots < F_\ell$, $\Phi^{(F_0, \dots, F_\ell) \oplus C}_e(x) \downarrow \in U$
and for every~$i \in (k, \ell]_{\N}$, $F_i \in \LL \cap Z_0$ or~$F_i \in \LL \cap Z_1$.
If such a set~$U$ is found, $f(x) =  1 + max(U)$, otherwise~$f(x) \uparrow$.
The function~$f$ is partial $X \oplus C$-computable. We have two cases.
\begin{itemize}
	\item Case 1: $f$ is total. By $X \oplus C$-hyperimmunity of~$B_i$,
	there is some~$x$ such that~$f(x) \leq p_{B_i}(x)$. Let~$U$ be the finite set
	witnessing~$f(x) \downarrow$. By taking~$Z_0 = X \cap A$ and~$Z_1 = X \cap \overline{A}$,
	there is a finite sequence of sets~$F_{k+1}, \dots, F_\ell$ such that
	$\Phi^{(F_0, \dots, F_\ell) \oplus C}_e(x) \downarrow \in U$
	and~$d = (F_0, \dots, F_\ell, X)$ is a valid extension of~$c$.
	The condition $d$ forces~$\Phi^{G \oplus C}_e(x) < f(x)$.

	\item Case 2: there is some~$x$ such that~$f(x) \uparrow$.
	By compactness, the $\Pi^{0, X \oplus C}_1$ class~$\mathcal{C}$ of sets~$Z_0 \oplus Z_1$
	such that~$Z_0 \cup Z_1 = X$ and $\Phi^{(F_0, \dots, F_\ell) \oplus C}_e(x) \uparrow$ 
	for every finite sequence of sets~$F_{k+1}, \dots, F_\ell$ such that
	$F_k < \dots < F_\ell$ and for every~$i \in (k, \ell]_{\N}$, $F_i \in \LL \cap Z_0$ or~$F_i \in \LL \cap Z_1$
	is not empty. By the hyperimmune-free basis theorem~\cite{Jockusch197201}, there exists some partition~$Z_0 \oplus Z_1 \in \mathcal{C}$
	such that the~$B$'s are $Z_0 \oplus Z_1 \oplus C$-hyperimmune. The set~$Z_j$ is infinite for some~$j < 2$
	and the condition $d = (F_0, \dots, F_k, Z_{j})$ is an extension of~$c$ forcing~$\Phi_e^{G \oplus C}(x) \uparrow$.
\end{itemize} 
\end{proof}

Let~$\mathcal{F}$ be a sufficiently generic filter for this notion of forcing.
The filter $\mathcal{F}$ yields a sequence~$\vec{F} = \langle F_0, F_1, \dots \rangle$ 
which is infinite by the first claim, and such that the~$B$'s are $\vec{F} \oplus C$-hyperimmune by the second claim.
This finishes the proof of Theorem~\ref{thm:sgp22-preservation-hyperimmunity}.
\end{proof}

\begin{corollary}
$\RCAo + \GP + \COH + \EM + \WKL$ does not imply $\ADS$. 
\end{corollary}
\begin{proof}
By the hyperimmune-free basis theorem~\cite{Jockusch197201}, $\WKL$ admits preservation of hyperimmunity.
The first author proved in~\cite{Patey2015Iterative} that~$\COH$ and $\EM$ admit preservation of hyperimmunity,
but that $\ADS$ does not. Last, $\GP$ admits preservation of hyperimmunity
since $\COH+\SGP$ implies~$\GP$ over~$\RCAo$.
\end{proof}

\begin{definition}[Diagonally non-computable function]
A function $f$ is \emph{diagonally non-computable} (d.n.c.) relative to~$X$ if for every~$e$, $f(e) \neq \Phi_e^X(e)$.
$\DNCZP$ is the statement ``For every set~$X$, there is a function d.n.c.\ relative to the jump of~$X$''.
\end{definition}

Beware, the notation $\DNCZP$ may cause some confusion with $\DNC_2$, the restriction to $\{0,1\}$-valued d.n.c.\ functions
which is equivalent to~$\WKL$ over~$\RCAo$. 
The following proof is an adaptation of the proof that the Erd\H{o}s-Moser implies~$\DNCZP$
over~$\RCAo$~\cite{Patey2015Somewhere}.

\begin{theorem}
$\RCA \vdash \GP(\LL_\bbomega) \rightarrow \DNCZP$
\end{theorem}
\begin{proof}
Fix a set~$X$. Let $g(.,.)$ be a total $X$-computable function 
such that $\Phi_e^{X'}(e)=\lim_s g(e,s)$ if the limit exists, 
and $\Phi^{X'}_e(e) \uparrow$ if the limit does not exist.
Also fix for each~$e \in \N$ an enumeration $D_{e,0}, D_{e,1}, \dots$ of all finite sets of cardinal $3^{e+1}$.
We define the function $f : [\N]^2 \to 2$ by primitive recursion.
Let~$f_0$ be the function nowhere defined.
At stage~$s+1$, do the following. Start with $f_{s+1}=f_s$. 
Then, for each $e<s$, take the first pair $\{x,y\} \in [(D_{e,g(e,s)} \cap [0,s)_{\N}) \setminus \bigcup_{k<e} D_{k,g(k,s)}]^{2}$ 
if it exists, and set~$f_{s+1}(x,s) = 0$ and~$f_{s+1}(y,s) = 1$. 
Finally, set $f(z, s) = 0$ for any $z<s$ such that $f_{s+1}(s,z)$ remains undefined.
This finishes the construction of $f_{s+1}$. 
Note that $f_{s}$ is defined on $[[0,s]_{\N}]^2$.
Thus, $f=\bigcup_s f_s$ must exist and is total on $[\N]^{2}$.

By $\GP(\LL_\bbomega)$, let $\vec{F} = \langle F_0, F_1, \dots \rangle$ 
be an infinite $\LL_\bbomega$-grouping for $f$.
Let $h(e)$ be such that $D_{e, h(e)} \subseteq F$ for some $F \in \vec{F}$. Such an~$F$
exists since $D_{e, 0}, D_{e,1}, \dots$ enumerates all finite sets of cardinality $3^{e+1}$,
and $\vec{F}$ contains sets of arbitrary size.
We claim that $h(e) \not=\Phi_e^{X'}(e)$ for all~$e$, which would prove $\DNCZP$. 
Suppose otherwise, i.e., suppose that $\Phi_e^{X'}(e)=h(e)$ for some~$e$. Let~$F \in \vec{F}$
be such that~$D_{e,h(e)} \subseteq F$.
Then there is a stage $s_0$ such that $h(e)=g(e,s)$ for all~$s \geq s_0$ or equivalently 
$D_{e,g(e,s)}=D_{e,h(e)} \subseteq F$ for all~$s \geq s_0$. We claim that for any $s$ be bigger than both 
$\max(F)$ and $s_0$, there are some~$x, y \in D_{e,h(e)} \subseteq F$ such that~$f(x,s) \neq f(y,s)$, which
contradicts the fact that~$\vec{F}$ is an $\LL_\bbomega$-grouping for~$f$.

To see this, let $s$ be such a stage. At that stage~$s$ of the construction of $f$, 
a pair $\{x,y\} \in [(D_{e,g(e,s)} \cap [0,s)_{\N}) \setminus \bigcup_{k<e} D_{k,g(k,s)}]^{2}$ is selected 
by a cardinality argument since $|D_{e,s} \cap [0,s)_{\N}| = |D_{e,s}| = 3^{e+1} > \sum_{k < e} 3^{k+1} = |\bigcup_{k<e} D_{k,g(k,s)}|$.
Since $D_{e,g(e,s)} = D_{e,h(e)} \subseteq F$, this pair is contained in $F$. 
At this stage, we set~$f(x,s) \neq f(y,s)$, therefore, $\vec{F}$ is not an $\LL_\bbomega$-grouping for~$f$, contradiction.
\end{proof}

In particular, $\SRT^2_2$ does not imply $\GP(\LL_\bbomega)$ over~$\RCAo$
since Chong, Slaman and Yang~\cite{CSY2014} built a (non-standard) model of~$\SRT^2_2 + \BII$
containing only low sets, whereas provably in $\RCAo$, 
there is no $\Delta^0_2$ d.n.c.\ function relative to~$\emptyset'$.

\begin{definition}[Rainbow Ramsey theorem]
Fix $n, k \in \N$. A coloring function $f: [\N]^n \to \N$ is \emph{$k$-bounded}
if for every $y \in \N$, $\card{f^{-1}(y)} \leq k$. A set $R$ is a \emph{rainbow} for $f$ (or an \emph{$f$-rainbow})
if $f$ is injective over~$[R]^n$.
$\RRT^n_k$ is the statement ``Every $k$-bounded function $f : [\N]^n \to \N$
has an infinite $f$-rainbow''.
\end{definition}

Miller~\cite{MillerAssorted} proved that the statement~$\DNCZP$ is equivalent to the rainbow Ramsey theorem
for pairs ($\RRT^2_2$) over~$\RCAo$. 

\begin{corollary}
$\RCA \vdash \GP \rightarrow \RRT^2_2$.
\end{corollary}

Seetapun and Slaman~\cite{Seetapun1995strength} defined a \emph{Cardinality scheme}
for a set of formulas~$\Gamma$ as follows. For every~$\varphi(x,y) \in \Gamma$,
$\mathrm{C}\Gamma$ contains the universal closure of the formula ``If $\varphi(x, y)$ defines
an injective function, then its range is unbounded''. 
Conidis and Slaman~\cite{Conidis2013Random} proved
that the rainbow Ramsey theorem for pairs implies the $\Sigma_2$ cardinality scheme ($\CSig_2$).

\begin{corollary}
$\RCA \vdash \GP \rightarrow \CSig_2$.
\end{corollary}

In particular, this shows that $\GP$ is not $\Pi^1_1$-conservative 
over~$\RCAo+\II$ since a Skolem hull argument shows
that $\CSig_2$ is not provable in~$\Ii$ (see Seetapun and Slaman~\cite{Seetapun1995strength}).

%\section{The strength of $\GP$}\label{section-GP-con}
\section{Conservation theorem for $\GP$}\label{section-GP-con}
In this section, we will prove a conservation result for the grouping principle.
To calculate the size of $\alpha$-large($\EM$) sets in Section~\ref{section-RT22}, we will use a finite version of the grouping principle within $\II$.
Instead of proving the finite grouping principle within $\II$ directly, we would like to show a conservation theorem for the infinite grouping principle over $\II$.

In Section~\ref{section-GP}, we have seen that $\SGP$ has an $\omega$-model with only low sets.
It is well-known that a low solution construction is often able to be converted into a forcing construction of a solution satisfying $\Sigma^{0}_{1}$-induction within a countable nonstandard model, which leads to a $\Pi^{1}_{1}$-conservation over $\II$.
Unfortunately, our construction of a low solution for $\SGP$ in Section~\ref{section-GP} requires $\BII$ and thus it is not formalizable within $\II$.
To overcome this situation, we show a general conservation theorem characterized by using recursively saturated models.
\begin{thm}\label{thm-con-Bn+1-vs-In}
Fix $n\ge 1$.
Let $\Gamma$ be a formula of the form $\A X \E Y \theta(X,Y)$ where $\theta$ is $\Pi^{0}_{n+1}$.
Then, $\RCAo+\BN[n+1]+\Gamma$ is a $\tilde\Pi^{0}_{n+2}$-conservative extension of $\IN$ if the following condition holds:
\begin{itemize}
  \item[$(\dag)$] for any countable recursively saturated model $(M,S)\models \BN[n+1]$ and for any $X\in S$, there exists $Y\subseteq M$ such that $(M,S\cup\{Y\})\models \IN+\theta(X,Y)$.
\end{itemize}
\end{thm}

To show this theorem, we use the following property of recursively saturated models and resplendent models, which are introduced by Barwise and Schlipf.
See \cite{Barwise-Schlipf} for the historical information of recursively saturated models and resplendent models.
\begin{thm}[see Sections~1.8 and 1.9 of \cite{Kossak-Schmerl}]\label{thm-c-resplendency}
Let $\mathcal{L}\supseteq \mathcal{L}_{\PA}$ be a finite language, and let $\mathcal{M}$ be a countable $\mathcal{L}$-structure. Then the following are equivalent.
\begin{enumerate}
 \item $\mathcal{M}$ is recursively saturated.
 \item $\mathcal{M}$ is resplendent, \textit{i.e.}, for any recursive set of sentences $T$ of a finite language $\mathcal{L}'\supseteq\mathcal{L}$ such that $\mathrm{Th}(\mathcal{M})\cup T$ is consistent, there exists an expansion $\mathcal{M}'$ of $\mathcal{M}$ such that $\mathcal{M}'\models T$.
 \item $\mathcal{M}$ is chronically resplendent, \textit{i.e.}, $\mathcal{M}$ is resplendent with the extra condition that the expansion $\mathcal{M}'$ is resplendent as an $\mathcal{L}'$-structure.
%  \item $\mathcal{M}$ is chronically resplendent, \textit{i.e.}, for any set of sentences $T$ of a finite language $\mathcal{L}'\supseteq\mathcal{L}$ such that $\mathrm{Th}(\mathcal{M})\cup T$ is consistent, there exists an expansion $\mathcal{M}'$ of $\mathcal{M}$ such that $\mathcal{M}'\models T$ and $\mathcal{M}'$ is recursively saturated as an $\mathcal{L}'$-structure.%
%\footnote{The usual definition of chronic resplendency is saying that $\mathcal{M}'$ is resplendent instead of recursively saturated, but in our situation, they are equivalent. See \cite{Kossak-Schmerl}.}
\end{enumerate}
\end{thm}

\begin{proof}[Proof of Theorem~\ref{thm-con-Bn+1-vs-In}.]
Let $\Gamma\equiv\A X \E Y \theta(X,Y)$ where $\theta$ is $\Pi^{0}_{n+1}$ satisfy the condition $(\dag)$, and let $\varphi_{0}\equiv\A X \A x\varphi(X,x)$ be a $\tilde\Pi^{0}_{n+2}$-sentence where $\varphi$ is $\Sigma^{0}_{n+1}$.
We will show that $\IN\not\vdash \varphi_{0}$ implies that $\RCAo+\BN[n+1]+\Gamma\not\vdash \varphi_{0}$.
Assume that $\IN\not\vdash \varphi_{0}$, and take a countable recursively saturated model $(M,S)\models\IN$ such that $(M,S)\models\neg\varphi_{0}$.
Then, there exists $a\in M$ and $A\in S$ such that $(M,\{A\})\models \neg\varphi(A,a)$.
We will construct an ($\omega$-length) sequence of cuts $M=I_{0}\supseteq_{e}I_{1}\supseteq_{e}\dots$ and a sequence of sets $A_{i}\subseteq I_{i}$ such that
\begin{itemize}
 \item $(I_{i},\{A_{0},\dots,A_{i}\}\rest I_{i})\models \IN$ and $(I_{i},\{A_{0},\dots,A_{i}\}\rest I_{i})$ is recursively saturated,
 \item if $i<j$, then $(I_{j},\{A_{0},\dots,A_{i}\}\rest I_{j})$ is a $\Sigma^{0}_{n}$-elementary substructure of $(I_{i},\{A_{0},\dots,A_{i}\}\rest I_{i})$, and,
 \item $(I_{i+1},\{A_{0},\dots,A_{i}, Z_{i},A_{i+1}\}\rest I_{i+1})\models \theta(Z_{i},A_{i+1})$, where $Z_{i}$ is a $k$-th $\Delta^{0}_{1}$-definable set in $(I_{j},\{A_{0},\dots,A_{j}\}\rest I_{j})$ if $i=(j,k)$.
\end{itemize}
Set $I_{0}=M$ and $A_{0}=A\oplus \{a\}$.
Now, given $(I_{i},\{A_{0},\dots,A_{i}\}\rest I_{i})$, we will first find a cut $I_{i+1}\subsetneq_{e}I_{i}$ such that $(I_{i+1},\{A_{0},\dots,A_{i}\}\rest I_{i+1})$ is $\Sigma^{0}_{n}$-elementary substructure of $(I_{i},\{A_{0},\dots,A_{i}\}\rest I_{i})$ and recursively saturated.
Let $J_{0}$ be the set of all $\Sigma^{0}_{n}$-definable elements in $(I_{i},\{A_{0},\dots,A_{i}\}\rest I_{i})$.
Since $(I_{i},\{A_{0},\dots,A_{i}\}\rest I_{i})$ is recursively saturated, $J_{0}$ is not cofinal in $I_{i}$, thus $J=\sup J_{0}$ forms a proper cut of $I_{i}$ and it is a $\Sigma^{0}_{n}$-elementary substructure.
Therefore, a recursive $\mathcal{L}_{\PA}\cup\{A_{0},\dots, A_{i}, J\}$-theory T saying that $(J,\{A_{0},\dots,A_{i}\}\rest J)$ is a $\Sigma^{0}_{n}$-elementary proper cut of $(I_{i},\{A_{0},\dots,A_{i}\}\rest I_{i})$ is consistent with $\Th((I_{i},\{A_{0},\dots,A_{i}\}\rest I_{i}))$. (One can state $\Sigma^{0}_{n}$-elementarity by using the truth predicate.)
Thus, by the chronic resplendency of Theorem~\ref{thm-c-resplendency}, there exists $J'\subseteq I_{i}$ such that $(I_{i},\{A_{0},\dots, A_{i},J'\}\rest I_{i})$ satisfies $T$ and is recursively saturated.
Let $I_{i+1}$ be such $J'$, then $(I_{i+1},\{A_{0},\allowbreak\dots,A_{i}\}\rest I_{i+1})$ is $\Sigma^{0}_{n}$-elementary substructure of $(I_{i},\{A_{0},\dots,A_{i}\}\rest I_{i})$ and recursively saturated.

By Theorem~\ref{thm-elem-cut-vs-Bn}, $(I_{i+1},\{A_{0},\dots,A_{i}\}\rest I_{i+1})\models\BN[n+1]$, and thus $(I_{i+1},\{A_{0},\dots,A_{i}, Z_{i}\}\rest I_{i+1})\models\BN[n+1]$.
($Z_{i}\cap I_{i+1}$ is $\Delta^{0}_{1}$-definable in $(I_{i+1},\{A_{0},\dots,A_{i}\}\rest I_{i+1})$ by $\Sigma^{0}_{n}$-elementarity.)
Thus, by the condition $(\dag)$, there exists $B\subseteq I_{i+1}$ such that $(I_{i+1},\{A_{0},\dots,A_{i}, Z_{i},B\}\rest I_{i+1})\models\theta(Z_{i},B)$.
By using chronic resplendency as above, one can re-choose $B\subseteq I_{i+1}$ so that $(I_{i+1},\{A_{0},\dots,A_{i}, Z_{i},B\}\rest I_{i+1})$ is recursively saturated.
Then, put $A_{i+1}$ as such $B$.

\begin{claim*}
$\bar{I}=\bigcap_{i\in\omega} I_{i}$ is a cut of $M$, and $(\bar{I},\{A_{0},\dots,A_{i}\}\rest \bar{I})$ is a $\Sigma^{0}_{n}$-elementary substructure of $(I_{i},\{A_{0},\dots,A_{i}\}\rest I_{i})$.
\end{claim*}
\begin{proof}[Proof of the claim.]
Clearly $\bar{I}$ forms a cut, and thus it is always a $\Sigma^{0}_{0}$-elementary substructure.
We show $\Sigma^{0}_{n}$-elementarity by induction on the complexity of formulas.
Let $0<k\le n$ and assume that $(I_{i},\{A_{0},\dots,A_{i}\}\rest I_{i})\models \E x\psi(x,c)$ where $c\in\bar{I}$ and $\psi$ is $\Pi^{0}_{k-1}$.
By $\IN$ in $(I_{i},\{A_{0},\dots,A_{i}\}\rest I_{i})$, take the least $d\in I_{i}$ such that $\psi(d,c)$ holds.
If $d\notin \bar{I}$, then there exists $j>i$ such that $d\notin I_{j}$, thus $(I_{j},\{A_{0},\dots,A_{i}\}\rest I_{j})\models \neg\E x\psi(x,c)$.
This contradicts the second condition of the construction of $I_{i}$'s, thus $d\in \bar{I}$.
Hence, $(\bar{I},\{A_{0},\dots,A_{i}\}\rest \bar{I})\models \psi(d,c)$ by the induction hypothesis.
\end{proof}

Put $\bar{S}=\Delta^{0}_{1}\text{-}\mathrm{Def}(\bar{I},\{A_{i}\mid i\in\omega\}\rest \bar{I})$.
Then, by the claim and Theorem~\ref{thm-elem-cut-vs-Bn}, $(\bar{I},\bar{S})\models\RCAo+\BN[n+1]$.
By the third condition of the construction of $A_{i}$'s and the $\Sigma^{0}_{n}$-elementarity, we have $(\bar{I},\bar{S})\models\Gamma+\neg\varphi(A,a)$.
Thus, $\RCAo+\BN[n+1]+\Gamma\not\vdash \varphi_{0}$.
\end{proof}

Now, we apply Theorem~\ref{thm-con-Bn+1-vs-In} to $\SGP$.
For this, we reformulate the low solution construction for $\SGP$ as follows.
\begin{lem}\label{lem:ISig1-preservation}
Let $(M,S_{0})$ be a countable model of $\BII$,
% such that $S_{0}$ is finite and $f:[\N]^{2}\to 2$ and $\LL\subseteq[\N]^{<\N}$ are $\Delta^{0}_{1}$-definable in $(M,S_{0})$.
 and let $f:[\N]^{2}\to 2$ and $\LL\subseteq[\N]^{<\N}$ be members of $S_{0}$.
Then, there exists $G\subseteq M$ such that 
\begin{itemize}
 \item[] $(M,S_{0}\cup\{G\})\models \II+$``$G$ is a witness that $\LL$ is not a largeness notion'', or,
 \item[] $(M,S_{0}\cup\{G\})\models \II+$``$G$ is an infinite $\LL$-grouping for $A_{f}=\{x\in \N\mid \lim_{y\to\infty}f(x,y)=1\}$''.
\end{itemize} 
\end{lem}
\begin{proof}
Let $(M,S_{0})$ be a countable model of $\BII$ such that
% $S_{0}$ is finite and $f:[\N]^{2}\to 2$ and $\LL\subseteq[\N]^{<\N}$ are $\Delta^{0}_{1}$-definable in $(M,S_{0})$.
$f,\LL\in S_{0}$.
By H\'ajek\cite{Hajek}, we can always find an $\omega$-extension $S\supseteq S_{0}$ such that $(M,S)\models \BII+\WKLo$ (see also Belanger\cite{Belanger}).
%Then, $f,\LL\in S$.
Thus, we will work on $(M,S)$ instead of $(M,S_{0})$.
If $\LL$ is not a largeness notion in $(M,S)$, take a witness $G\in S$ of not being a largeness notion, then, we have done.
Otherwise, we will construct an infinite $\LL$-grouping for $A_{f}$.

The following construction is a ``model-theoretic interpretation'' of Theorem~\ref{thm:grouping-low-solution}.
To simplify the coding, we will only consider a minimal $\LL$-sequence, 
namely, a sequence of the form $\langle F_{i}\in\LL\mid i<k \rangle$, 
$F_{i}\subseteq X_{0}$, $F_{i}<F_{j}$ if $i<j$ and 
$F_{i}\setminus\{\max{F_{i}}\}\notin \LL$, \textit{i.e.}, each $F_{i}$ is minimal.
A code for a minimal $\LL$-sequence is a binary sequence $\sigma\in 2^{<M}$ 
(which is coded in $M$) such that $\{x<|\sigma|\mid \sigma(x)=1\}=\bigcup_{i<k}F_{i}$ 
for some minimal $\LL$-sequence $\langle F_{i}\mid i<k \rangle$.
(By the minimality, one can effectively decode a binary sequence to obtain the $\LL$-sequence.)
Note that $\sigma$ may code extra 0's after $\max F_{k-1}$. 
Thus, one can identify a code $\sigma$ with a pair $(\langle F_{i}\mid i<k \rangle, d)$ where $d=|\sigma|$.
With this identification, we let $\|\sigma\|_{\LL}=k$.
Given an ($M$-)finite sequence of sets $\vec{Y}=\langle Y_{j}\mid j<l \rangle\in S$ 
and an ($M$-)finite set $F$, $F$ is said to be consistent with $\vec{Y}$ 
if for any $j<l$, $(F\subseteq Y_{j}\vee F\subseteq \overline{Y_{j}})$.
A code $\sigma$ for a minimal $\LL$-sequence $\langle F_{i}\mid i<k \rangle$ 
is said to be consistent with $\vec{Y}$ if for any $i<k$, $F_{i}$ is consistent with $\vec{Y}$.
Given $t\in M$, we let $A_{f,t}=\{x< t\mid x\in X_{0}, \lim_{y\to \infty,y\in X_{0}}f(x,y)=1\}$. 
Note that $A_{f,t}\in S$ since $(M,S)\models\BII$.

Now, we will construct $G\subseteq M$ by arithmetical forcing. 
Let $\mathbb{P}$ be the set of all pairs of the form $(\sigma,\vec{Y})$ such that
\begin{itemize}
 \item $\vec{Y}=\langle Y_{j}\mid j<l \rangle$ is an ($M$-)finite sequence of sets in $S$,
 \item $\sigma$ is a code for a minimal $\LL$-sequence 
	which is consistent with $\vec{Y}$ and $A_{f,|\sigma|}$,
\end{itemize}
and we let $(\sigma,\vec{Y})\preceq(\tau,\vec{Z})$ 
if $\sigma\supseteq \tau$ and $\vec{Y}\supseteq\vec{Z}$.
Take an $(M,S)$-generic filter $\mathcal{G}$ of $(\mathbb{P},\preceq)$ 
and put $G=\bigcup\{\sigma\mid \E \vec{Y}\in S\, (\sigma,\vec{Y})\in \mathcal{G}\}$.
Then, this $G$ is the desired.
It is clear by construction that $G$ is a minimal $\LL$-sequence which is consistent with $A_{f}$.
To see that $G$ preserves $\II$ and $G$ is infinite in $(M,S)$, 
we need to check the following are dense in $(\mathbb{P},\preceq)$:
\begin{align*}
\mathcal{D}^{1}_{\theta,b}:=&\left\{(\sigma,\vec{Y})\,\Bigg|\,
\begin{aligned}
 &\quad \A \tau\supseteq\sigma(\tau\text{ is a minimal $\LL$-sequence 
	consistent with }\vec{Y}\to \A n\le |\tau|\neg\theta(b,n,\tau\rest n))\\
 &\vee \E c\le b(\A d<c\A \tau\supseteq\sigma(\tau\text{ is a minimal 
	$\LL$-sequence consistent with }\vec{Y}\\
 &\quad\quad\quad\quad   \to \A n\le |\tau|\neg\theta(d,n,\tau\rest n))
	\wedge \E n\le |\sigma| \theta(c,n, \sigma\rest n)) 
\end{aligned}
\right\},\\
&\mbox{where $b\in M$ and $\theta(x,n,\sigma)\in \Sigma^{0}_{0}$ with parameters from $(M,S)$},\\
\mathcal{D}^{2}_{e}:=&\left\{(\sigma,\vec{Y})\mid \|\sigma\|_{\LL}\ge e
\right\},
%\\&
\mbox{where $e\in M$}.
\end{align*}
One can easily see that if $(\sigma,\vec{Y})\in \mathcal{D}^{1}_{\theta,b}$, 
then $(\sigma,\vec{Y})$ forces ``if $\E n\theta(b,n,G\rest n)$, 
there exists least $c\le b$ such that $\E n\theta(c,n,G\rest n)$'', 
which guarantees $\Sigma^{0}_{1}$-least number principle, 
and if $(\sigma,\vec{Y})\in \mathcal{D}^{2}_{e}$, then $(\sigma,\vec{Y})$ forces 
``$\|G\|_{\LL}\ge e$'', which means $G$ is an infinite minimal $\LL$-sequence.

To see that $\mathcal{D}^{1}_{\theta,b}$ is dense, let $(\sigma,\vec{Y})$ be given.
Let $\Theta(x)$ be a $\Sigma^{0}_{1}$-formula saying that ``there exists $t$ 
such that for any $\vec{Z}=\langle Z_{j}\subseteq [|\sigma|,t)_{\N}\mid j<2^{x} \rangle$ 
there exists $\tau\supseteq\sigma$ with $|\tau|\le t$ such that $(\tau$ is a minimal 
$\LL$-sequence consistent with $\vec{Y}^{\frown}\vec{Z}\wedge \E n\le |\tau|\theta(x,n,\tau\rest n))$.''
We consider the two cases.
\begin{description}
 \item[Case I] $\Theta(b)$ fails in $(M,S)$.
\end{description}
In this case, by $\WKLo$ in $(M,S)$, there exists $\vec{Z}=\langle Z_{j}\subseteq 
[|\sigma|,\infty)_{\N}\mid j<2^{b} \rangle$ such that $\A \tau\supseteq\sigma(\tau$ 
is a minimal $\LL$-sequence consistent with $\vec{Y}^{\frown}\vec{Z}\to \A n\le |\tau|\neg\theta(b,n,\tau\rest n))$.
Take such $\vec{Z}$. Then, $(\sigma,\vec{Y}^{\frown}\vec{Z})\in \mathcal{D}^{1}_{\theta,b}$.
\begin{description}
 \item[Case II] $\Theta(b)$ holds in $(M,S)$.
\end{description}
In this case, by $\II$ in $(M,S)$, there exists the least $c\le b$ such that $\Theta(c)$ holds.
Then, by $\WKLo$, there exists $\vec{W}^{d}=\langle W^{d}_{j}\subseteq 
[|\sigma|,\infty)_{\N}\mid j<2^{d} \rangle$ such that $\A \tau\supseteq\sigma(\tau$ 
is a minimal $\LL$-sequence consistent with 
$\vec{Y}^{\frown}\vec{W}^{d}\to \A n\le |\tau|\neg\theta(d,n,\tau\rest n))$ for any $d<c$.
Now, take the witness $t\in M$ for $\Theta(c)$, and put 
$\vec{Z}=\langle W^{d}_{j}\cap[|\sigma|,t)_{\N}\mid j<2^{d},d<c \rangle^{\frown}\langle A_{f,t}\cap[|\sigma|,t)_{\N} \rangle$.
Then, by $\Theta(c)$, there exists $\tau\supseteq\sigma$ 
with $|\tau|\le t$ such that $\tau$ is a minimal $\LL$-sequence 
consistent with $\vec{Y}^{\frown}\vec{Z}$ and $\E n\le |\tau|\theta(c,n,\tau\rest n)$.
Then $(\tau,\vec{Y}^{\frown}\langle W^{d}_{j}\mid j<2^{d},d<c \rangle)$ 
is a condition in $\mathbb{P}$ and $(\tau,\vec{Y}^{\frown}\langle 
W^{d}_{j}\mid j<2^{d},d<c \rangle)\in \mathcal{D}^{1}_{\theta,b}$.

To see that $\mathcal{D}^{2}_{e}$ is dense, let $(\sigma,\vec{Y})$ 
be given where $\sigma$ is a code for $\langle F_{i}\mid i<k \rangle$ 
and $\vec{Y}=\langle Y_{j}\mid j<l \rangle$.
By applying Lemma~\ref{lem-largeness} $e$ times in $(M,S)$, 
one can find $\langle t_{0},\dots, t_{e} \rangle$ such that 
$t_{0}=|\sigma|$ and for any $s<e$ and for any $2^{l+1}$ splitting 
$M=\bigsqcup_{p<2^{l+1}}W_{p}$, there exists a finite subset 
of $[t_{s},t_{s+1})_{\N}$ which is $\LL$-large and included in one of $W_{p}$.
($\II$ is enough for this iteration.)
Thus, there exists $E_{s}\subseteq [t_{s},t_{s+1})_{\N}$ which is 
$\LL$-large and consistent with $\vec{Y}$ and $A_{f,t_{e}}$ for any $s<e$.
Let $\tau\supseteq \sigma$ be a code for a sequence 
$\langle F_{i}\mid i<k \rangle^{\frown}\langle E_{s}\mid s<e \rangle$.
Then, $(\tau,\vec{Y})$ is a condition in $\mathbb{P}$ and 
$\|\tau\|_{\LL}\ge e$, thus, $(\tau, \vec{Y})\in \mathcal{D}^{2}_{e}$.
\end{proof}

\begin{thm}\label{thm:SGP-conservation}
$\RCAo+\SGP$ is a $\tilde\Pi^{0}_{3}$-conservative extension of $\II$.
\end{thm}
\begin{proof}
Straightforward from Lemma~\ref{lem:stable-grouping-characterization}, Theorem~\ref{thm-con-Bn+1-vs-In} and Lemma~\ref{lem:ISig1-preservation}.
\end{proof}

Thus, by the amalgamation theorem, we have the following conservation result.
\begin{thm}\label{thm:GP-conservation}
$\WKLo+\GP+\ADS$ is a $\tilde\Pi^{0}_{3}$-conservative extension of $\II$.
\end{thm}
\begin{proof}
Apply Theorem~\ref{thm:amalgamation} for the conservation results Corollary~\ref{cor-con-ADS} and Theorem~\ref{thm:SGP-conservation}, then we can see that $\WKLo+\ADS+\SGP$ is a $\tilde\Pi^{0}_{3}$ conservative extension of $\II$.
By Hirschfeldt and Shore~\cite{HS2007}, $\ADS$ implies $\COH$.
Thus, $\WKLo+\GP+\ADS$ is a $\tilde\Pi^{0}_{3}$-conservative extension of $\II$.
\end{proof}

\section{Conservation theorem for $\RT^2_2$}\label{section-RT22}

In this section, we will show that $\WKLo+\RT^2_2$ is a $\tilde{\Pi}^0_3$-conservative 
extension of $\II$
by showing that $\WKLo+\EM$ is a $\tilde{\Pi}^0_3$-conservative 
extension of $\II$.
For this, we will bound the size of $\bbomega^{k}$-large($\EM$) sets 
by $\bbomega^{n}$-large sets using the following finite grouping principle.

\begin{definition}[$\II$, finite grouping principle]
Given largeness notions $\LL_{1}, \LL_{2}$ and a coloring 
$f:[X]^{n}\to k$, \textit{$(\LL_{1},\LL_{2})$-grouping for $f$} 
is a finite family of finite sets $\langle F_{i}\subseteq X\mid i<l \rangle$ such that 
\begin{itemize}
 \item $\A i<j<l\, F_{i}<F_{j}$,
 \item $\A i<l\, F_{i}\in \LL_{1}$,
 \item for any $H\subseteq_{\fin}\N$, if $H\cap F_{i}\neq\emptyset$ for any $i<l$, then $H\in \LL_{2}$, and,
 \item $\A i_{1}<\dots< i_{n}\, \E c<k\, \A x_{1}\in F_{i_{1}},\dots,\A x_{n}\in F_{i_{n}}\, f(x_{1},\dots,x_{n})=c$.
\end{itemize}
Note that if $\LL_{2}$ is regular, then the third condition 
can be replaced with $\{\max F_{i}\mid i<l\}\in \LL_{2}$.
Now $\FGPg^{n}_{k}(\LL_{1},\LL_{2})$ (finite grouping principle 
for $(\LL_{1},\LL_{2})$) asserts that for any infinite set $X_{0}\subseteq\N$, 
there exists a finite set $X\subseteq X_{0}$ such that for any 
coloring $f:[X]^{n}\to k$, there exists a $(\LL_{1},\LL_{2})$-grouping for $f$.
%We write $\FGPg^{n}_{k}$ for the statement saying that for any largeness notions 
%$\LL_{1}$, $\LL_{2}$, $\FGPg^{n}_{k}(\LL_{1},\LL_{2})$ holds, $\FGPg^{n}$ 
% for $\A k\FGPg^{n}_{k}$, and $\FGPg$ for $\A n\FGPg^{n}$.
\end{definition}
\begin{thm}\label{thm-FGP}
Let $\LL_{1}$ and $\LL_{2}$ be $\Delta_{0}$-definable 
regular largeness notions provably in $\II$. 
Then, $\II$ proves $\FGP(\LL_{1},\LL_{2})$.
\end{thm}
\begin{proof}
One can easily check that $\FGP(\LL_{1},\LL_{2})$ is a $\tilde\Pi^{0}_{3}$-statement which is provable from $\WKLo+\GP(\LL_{1})$.
Thus, $\II$ proves $\FGP(\LL_{1},\LL_{2})$ by Theorem~\ref{thm:GP-conservation}.
\end{proof}

Now we apply the generalized Parsons theorem to the finite grouping principle.
Actually, its upper bound also bounds $\bbomega^{k}$-large($\EM$) sets as follows.
\begin{lem}\label{lem-EM-largeness}
For any $k\in\omega$, there exists $n\in\omega$ such that
\begin{itemize}
 \item[] \rm $\II\vdash \A Z\subseteq_{\fin}(3,\infty)_{\N}(Z$ 
is $\bbomega^{n}$-large $\to Z$ is $\bbomega^{k}$-large($\EM$)$)$.
\end{itemize}
\end{lem}
\begin{proof}
We will prove this by (external) induction.
For the case $k=1$, $n=6$ is enough by Theorem~\ref{SK-theorem}.
Assume now $k>1$ and $\bbomega^{n_{0}}$-largeness implies 
$\bbomega^{k-1}$-large($\EM$)ness in $\II$.
By Corollary~\ref{cor-g-Parsons} and Theorem~\ref{thm-FGP},
take $n\in\omega$ so that 
$\II$ proves $\A Z\subseteq_{\fin}\N(Z$ is $\bbomega^{n}$-large $\to$ 
any coloring $f:[Z]^{2}\to 2$ has an $(\bbomega^{n_{0}},\bbomega^{6})$-grouping$)$.
Within $\II$, given an $\bbomega^{n}$-large set $Z\subseteq (3,\infty)_{\N}$ 
and $f:[Z]^{2}\to 2$, we want to find $H\subseteq Z$ such that 
$f$ is transitive on $[H]^{2}$ and $H$ is $\bbomega^{k}$-large.
By the assumption, there exists an $(\bbomega^{n_{0}},\bbomega^{6})$-grouping 
$\langle F_{i}\mid i<l \rangle\subseteq Z$ for $f$.
Since each $F_{i}$ is $\bbomega^{n_{0}}$-large, it is $\bbomega^{k-1}$-large($\EM$), 
thus, there exists $H_{i}\subseteq F_{i}$ such that $H_{i}$ 
is $\bbomega^{k-1}$-large and $f$ is transitive on $[H_{i}]^{2}$. 
On the other hand, $\{\max F_{i}\mid i<l\}$ is $\bbomega^{6}$-large, 
thus, there exists $\tilde{H}\subseteq\{\max F_{i}\mid i<l\}$ 
such that $\tilde{H}$ is $\bbomega$-large and $f$ is constant on 
$[\tilde{H}]^{2}$ by Theorem~\ref{SK-theorem}.
Put $H=\bigcup\{H_{i}\mid i<l, \max F_{i}\in \tilde{H}\}$.
Then, one can easily check that $H$ is $\bbomega^{k}$-large.
We now show that $f$ is transitive on $[H]^{2}$. Let $a,b,c\in H$ and 
$a<b<c$. If there exists $i<l$ such that $a,b,c\in H_{i}$, 
then $f$ is transitive for $a,b,c$ since $f$ is transitive on $[H_{i}]^{2}$.
If for some $i_{0}<i_{1}<l$, $a,b\in H_{i_{0}}$ and $c\in H_{i_{1}}$, 
then, $f(a,c)=f(b,c)$, so $f$ is transitive for $a,b,c$.
The case $a\in H_{i_{0}}$ and $b,c\in H_{i_{1}}$ for some $i_{0}<i_{1}<l$ is similar.
Finally, if for some $i_{0}<i_{1}<i_{2}<l$, $a\in H_{i_{0}}$, 
$b\in H_{i_{1}}$ and $c\in H_{i_{2}}$, then $f(a,b)=f(\max F_{i_{0}},
\max F_{i_{1}})=f(\max F_{i_{0}},\max F_{i_{2}})=f(a,c)$, 
thus $f$ is transitive for $a,b,c$.
\end{proof}

\begin{thm}\label{thm-con-EM}
$\WKLo+\EM$ is a $\tilde{\Pi}{}^{0}_{3}$-conservative extension of $\II$.
\end{thm}
\begin{proof}
By Theorems~\ref{thm-omega-largeness}, 
\ref{thm-provably-large-conservative} and Lemma~\ref{lem-EM-largeness}.
\end{proof}

Now the main theorem follows from the amalgamation theorem.
\begin{thm}\label{thm-con-RT22}
$\WKLo+\RT^{2}_{2}$ is a $\tilde{\Pi}{}^{0}_{3}$-conservative extension of $\II$.
\end{thm}
\begin{proof}
By Theorem~\ref{thm:amalgamation}, Corollary~\ref{cor-con-ADS} and Theorem~\ref{thm-con-EM}.
\end{proof}

Seetapun and Slaman~\cite{Seetapun1995strength} asked whether $\RCAo + \RT^2_2$ proves the consistency of $\II$,
and Cholak, Jockusch and Slaman~\cite{CJS} asked whether $\RCAo + \RT^2_2$ proves the totality of Ackermann function.
We answer negatively through the main theorem.
\begin{cor}\label{cor-answer1}
$\WKLo+\RT^{2}_{2}$ does not imply the consistency of $\II$ nor the totality of Ackermann function. 
\end{cor}
Chong and Yang~\cite{CY2015} asked what the proof-theoretic ordinal of $\RCAo+\RT^{2}_{2}$ is.
We again answer this question through the main theorem.
\begin{cor}\label{cor-answer-2}
The proof-theoretic ordinal of $\RCAo+\RT^{2}_{2}$ or $\WKLo+\RT^{2}_{2}$ is $\omega^{\omega}$.
\end{cor}

Note that one can avoid using the amalgamation theorem by directly 
combining the bounds for large($\psRT^{2}_{2}$)ness and large($\EM$)ness
in order to obtain a bound for  $\bbomega^{k}$-large($\RT^{2}_{2}$) sets and reprove the main conservation theorem.
\begin{prop}\label{prop-RT22-largeness}
For any $k\in\omega$, there exists $n\in\omega$ such that
\begin{itemize}
 \item[] \rm $\II\vdash \A Z\subseteq_{\fin}(3,\infty)_{\N}(Z$ 
is $\bbomega^{n}$-large $\to Z$ is $\bbomega^{k}$-large($\RT^{2}_{2}$)$)$.
\end{itemize}
\end{prop}
\begin{proof}
Given $k\in\omega$, take $n\in\omega$ so that $\II$ proves 
$\A Z\subseteq_{\fin}(3,\infty)_{\N}(Z$ is $\bbomega^{n}$-large $\to Z$ 
is $\bbomega^{2k+6}$-large($\EM$)$)$ by Lemma~\ref{lem-EM-largeness}.
Then, within $\II$, given $Z\subseteq_{\fin}(3,\infty)_{\N}$ 
which is $\bbomega^{n}$-large and $f:[Z]^{2}\to 2$, there exists $H_{0}\subseteq Z$ 
such that $H_{0}$ is $\bbomega^{2k+6}$-large and $f$ is transitive on $[H_{0}]^{2}$.
%One can easily find a transitive coloring $f':[[0,\max H_{0}]_{\N}]^{2}\to 2$ such that $f'\rest[H_{0}]^{2}=f\rest[H_{0}]^{2}$.
%(Set $f'(x,y)=0$ if $x\notin H_{0}$, $f'(x,y)=1$ if $x\in H_{0}$ and $y\notin H_{0}$, and $f'(x,y)=f(x,y)$ if $x,y\in H_{0}$.)
Then, by Lemma~\ref{lem-ADS-largeness}, there exists $H\subseteq H_{0}$ 
such that $H$ is $\bbomega^{k}$-large and $f$ is constant on $[H]^{2}$.
Thus, $Z$ is $\bbomega^{k}$-large($\RT^{2}_{2}$).
\end{proof}
Then, Theorem~\ref{thm-con-RT22} follows from Theorems~\ref{thm-omega-largeness}, \ref{thm-provably-large-conservative} and Proposition~\ref{prop-RT22-largeness}.

\section{Formalizing the conservation proof}\label{section-consistency-strength}

In this section, we will formalize Theorem~\ref{thm-con-RT22} within $\PRA$.
Actually, most arguments we used are straightforwardly formalizable within $\WKLo$.
We however need to take care of the use of external induction and non-computable construction of models.
We fix a standard provability predicate $\vdash$.

\begin{lem}\label{lem-formalization}
The following are provable within $\WKLo$.
\begin{enumerate}
\renewcommand{\labelenumi}{\rm(\arabic{enumi})}
 \item\label{e-0} $\A \varphi\in \Pi^{1}_{1}((\WKLo\vdash \varphi)\to (\II\vdash \varphi))$ 
		(Theorem~\ref{thm:harrington-conservation-wkl}).
 \item\label{e-2} $\A n\in\N (\II\vdash \mbox{any infinite set has an $\bbomega^{n}$-large subset})$ (Theorem~\ref{thm-omega-largeness}).
 \item\label{e-2-1} $\A k\in\N (\II\vdash \mbox{$X$ is $\bbomega^{k+4}$-large}\wedge \min X>3 \to \mbox{$X$ is $\bbomega$-large$(\RT^{2}_{k})$})$ (Theorem~\ref{SK-theorem}).
 \item\label{e-1-1} The generalized Parsons theorem (Corollary~\ref{cor-g-Parsons}).
 \item\label{e-1-2} The amalgamation theorem (Theorem~\ref{thm:amalgamation}).
 \item\label{e-2-2} $(\A n\in\N(\II\vdash \mbox{any infinite set has an $\bbomega^{n}$-large$(\Gamma)$ subset}))\to (\A \varphi\in \tilde\Pi^{0}_{3}((\WKLo+\Gamma\vdash \varphi)\to (\II\vdash \varphi)))$ for $\Gamma\equiv\psRT^{2}_{2}, \EM$ (Theorem~\ref{thm-provably-large-conservative}).
 \item\label{e-3-1} $\A n\in\N(\II\vdash \mbox{any infinite set has an $\bbomega^{n}$-large$(\psRT^{2}_{2})$ subset})$ (Lemma~\ref{lem-ADS-largeness}).
 \item\label{e-4-0} $\A \varphi\in \tilde\Pi^{0}_{3}((\WKLo+\GP\vdash \varphi)\to (\II\vdash \varphi))$ (Theorem~\ref{thm:GP-conservation}).
% \item\label{e-2-2} The conservation through largeness theorem (Theorem~\ref{thm-provably-large-conservative}).
% \item\label{e-3-1} The conservation theorem for $\ADS$ (Corollary~\ref{cor-con-ADS}).
% \item\label{e-4-0} The conservation theorem for $\GP$ (Theorem~\ref{thm:GP-conservation}).
 \item\label{e-4} $\A k\in \N(\II\vdash \FGP(\LL_{\bbomega^{k}},\LL_{\bbomega^{6}}))$ (Theorem~\ref{thm-FGP}).
 \item\label{e-5-1} $\A n\in\N(\II\vdash \mbox{any infinite set has an $\bbomega^{n}$-large$(\EM)$ subset})$ (Lemma~\ref{lem-EM-largeness}).
% \item\label{e-5-1} The conservation theorem for $\EM$ (Theorem~\ref{thm-con-EM}).
\end{enumerate}
\end{lem}
\newcommand\eref[1]{\textrm{(\ref{#1})}}
\begin{proof}
We reason within $\WKLo$.
For \eref{e-0}, several formalized proofs are known. See, e.g., \cite{Avigad1996,Hajek}.
For \eref{e-2}, the induction used here is on provability, thus it is a $\Sigma^{0}_{1}$-induction.
For \eref{e-2-1}, the original Ketonen and Solovay's proof is directly formalizable (see \cite[Section 6]{KS81}).
For \eref{e-1-1} and \eref{e-1-2}, we can directly formalize our model-theoretic proofs of Corollary~\ref{cor-g-Parsons}
 and Theorem~\ref{thm:amalgamation} by using the completeness theorem which is available within $\WKLo$.
For \eref{e-2-2}, we can formalize the proofs of Lemma~\ref{lem-cut-from-dense-set} and Theorem~\ref{thm-conservation-via-density-general} by using the completeness theorem.
To formalize the proof of Theorem~\ref{thm-provably-large-conservative}, we use the induction on provability.
For \eref{e-3-1}, formalize the proof of Lemma~\ref{lem-ADS-largeness}. Formalization is direct since we only deal with finite objects.
For \eref{e-4-0}, formalize the argument in Section~\ref{section-GP-con}.
To formalize the proof of Theorem~\ref{thm-con-Bn+1-vs-In}, an $\omega$-extension to be a model of $\BII+\WKLo$ is available within $\WKLo$ by formalizing the argument by H\'ajek\cite{Hajek} or Belanger\cite{Belanger}.
The existence of a countable recursively saturated model is provable in $\WKLo$ (see \cite[Section IX]{SOSOA}) and the Theorem~\ref{thm-c-resplendency} can be formalized similarly.
To formalize the proof of Lemma~\ref{lem:ISig1-preservation}, one can take a generic by the Baire category theorem which is available within $\RCAo$.
\eref{e-4} is straightforward from \eref{e-4-0}.
For \eref{e-5-1}, formalize the proof of Lemma~\ref{lem-EM-largeness}. The induction used here is again on provability.
\end{proof}

Thus, we have the following formalized conservation theorem.
\begin{thm}\label{thm:equi-consistency}
 $\PRA$ proves that $\WKLo+\RT^{2}_{2}\vdash\psi$ 
implies $\II\vdash \psi$ for any $\tilde{\Pi}{}^{0}_{3}$-sentences.
\end{thm}

Now the consistency equivalence of $\II$ and $\WKLo+\RT^{2}_{2}$ follows from this formalized conservation theorem.
\begin{cor}
 Over $\PRA$, $\Con(\II)$ is equivalent to $\Con(\WKLo+\RT^{2}_{2})$.
\end{cor}

\section{Open questions}

In their paper~\cite{CSY201X}, Chong, Slaman and Yang asked whether $\RT^2_2$ is a $\Pi^1_1$-conservative extension of $\RCAo + \BII$.
This question remains open, and a positive answer would strengthen our main conservation result since
$\BII$ is $\Pi^0_3$-conservative over~$\II$ (Theorem~\ref{thm:bsig2-pi03-conservative-isig1}).

\begin{question}[Chong, Slaman, Yang]\label{quest:rt22-pi11-bsig2}
Is $\RT^2_2$ a $\Pi^1_1$-conservative extension of~$\RCAo + \BII$?
\end{question}

In particular, they proved~\cite{CSY2012} that the chain anti-chain principle is $\Pi^1_1$-conservative over~$\RCAo + \BII$.
Therefore, in order to answer Question~\ref{quest:rt22-pi11-bsig2}, one needs only to prove that 
this is also the case for the Erd\H{o}s-Moser theorem.

For the purposes of our conservation proof, we introduced the grouping principle,
which seems to be interesting to study in its own right. 
%The standard proofs that a principle implies $\BII$ does not seem to be directly
%applicable to the grouping principle. The relation between the grouping
%principle and the $\Sigma^0_2$-bounding principle is currently unknown.
First, what is the first-order strength of $\GP$?
%\begin{question}
%Does $\GP$ imply~$\BII$ over~$\RCAo$?
%\end{question}
Alexander Kreuzer~\cite{Kre-private} gave a partial answer to this question
by proving the following theorem.
%The following proof is due to Kreuzer~\cite{Kre-private}.
\begin{thm}[Kreuzer]\label{thm:sgp-bii}
Over $\RCAo$, $\SGP$ implies $\BII$.
\end{thm}
\begin{proof}
Assume that $\BII$ fails.
As in Remark~\ref{rem:on-lem-largeness}, there exist a partition $X=X_{0}\sqcup\dots\sqcup X_{k-1}$ such that each of the $X_{i}$'s is finite, and then $\LL=\{F\in [\N]^{<\N}\mid \A i<k(F\not\subseteq X_{i})\}$ is a largeness notion.
Define $f:[\N]^{2}\to 2$ as $f(x,y)=1\leftrightarrow \E i<k(x\in X_{i}\wedge y\in X_{i})$.
Since $X_{i}$'s are all finite, $f$ is a stable coloring.
By $\SGP$, there exists an infinite $\LL$-grouping $\langle F_{j}\mid j\in\N \rangle$.
Since each of $F_{j}$ is not included in any of $X_{i}$'s, for any $j<j'$, the color between $F_{j}$ and $F_{j'}$ is $0$.
Thus, $\min F_{0},\dots,\min F_{k}$ are in different $X_{i}$'s, which is a contradiction.
\end{proof}

%A proof of $\Pi^1_1$-conservation of the grouping principle over~$\BII$
%would significantly simplify our conservation argument since the finite grouping
%principle can be easily derived from its infinite version.

Still, the following questions are remained open.
\begin{question}
For some $k\in\omega$, does $\GP(\LL_{\bbomega^{k}})$ imply~$\BII$ over~$\RCAo$?
\end{question}
\begin{question}
Is $\GP$ a $\Pi^1_1$-conservative extension of $\RCAo + \BII$?
\end{question}

The grouping principle has been used to establish a density bound for the Erd\H{o}s-Moser theorem.
The stable grouping principle does not imply the stable version of the Erd\H{o}s-Moser theorem
since the former admits low solutions whereas the latter does not. It is however unknown
whether the full version of the two principles coincide.

\begin{question}
Does~$\EM$ imply $\GP$ over~$\RCAo$?
\end{question}

Our conservation proof contains almost no information about the size of the proof, but it is interesting to know whether $\RT^{2}_{2}$ gives shorter proofs for $\tilde{\Pi}{}^{0}_{3}$-consequences of $\II$ or not.
\begin{question}
 Does $\WKLo+\RT^{2}_{2}$ or $\WKLo+\psRT^{2}_{2}$ have a significant speed-up over $\II$?
\end{question}
Note that there is no significant speed-up between $\WKLo$ and $\II$ (see Avigad\cite{Avigad1996}).
A killer example for this question is an existence of $m$-dense sets, \textit{i.e.}, $m\mbox{-}\PHt(\RT^{2}_{2})$ or $m\mbox{-}\PHt(\psRT^{2}_{2})$.
It is not hard to see that $m\mbox{-}\PHt(\RT^{2}_{2})$ can be proved from $\WKLo+\RT^{2}_{2}$ by using $\RT^{2}_{2}$ $m$-times, and the case of $\psRT^{2}_{2}$ is similar.
On the other hand, to prove $m\mbox{-}\PHt(\psRT^{2}_{2})$ from $\II$, what we know is the following.
\begin{prop}\label{thm-ADS-density}
For any $m\in\omega$, the following is provable within $\II$. 
If a finite set $X\subseteq\N$ is $\bbomega^{3^{m+1}}$-large and 
$\min X>3$, then $X$ is $m$-dense($\psRT^{2}_{2}$).
\end{prop}
\begin{proof}
Easy induction by using Lemma~\ref{lem-ADS-largeness}.
\end{proof}
\noindent
Thus, within $\II$, we can obtain an $m$-dense$(\psRT^{2}_{2})$-set by using $\Sigma^{0}_{1}$-induction $3^{m+1}$-times.
This might indicate that there is at most exponential speed-up for the case of $\psRT^{2}_{2}$.
For the case of $\RT^{2}_{2}$, the situation is more difficult.
An only upper bound for $m$-dense$(\RT^{2}_{2})$ sets obtained from our argument is the following.
\begin{prop}
 There exists a primitive recursive function $h:\omega\to\omega$ such that for any $m\in\omega$, the following is provable within $\II$. If a finite set $X\subseteq\omega$ is $\bbomega^{h(m)}$-large then $X$ is $m$-dense$(\RT^{2}_{2})$.
\end{prop}
\begin{proof}
By formalizing the proof of Proposition~\ref{prop-RT22-largeness} and applying $\Sigma^{0}_{1}$-induction, we obtain
\[\A m\in\N \E k\in\N (\II\vdash \mbox{$X$ is $\bbomega^{k}$-large}\to\mbox{$X$ is $m$-dense$(\RT^{2}_{2})$}).\]
Then, by the Parsons theorem, there exists a primitive recursive function $h:\omega\to \omega$ such that $\II$ proves
\[\A m\in\N \E k\le h(m)  (\II\vdash_{h(m)} \mbox{$X$ is $\bbomega^{k}$-large}\to\mbox{$X$ is $m$-dense$(\RT^{2}_{2})$}),\]
where $\vdash_{x}$ means that there exists a proof whose G\"odel number is smaller than $x$.
For $m\in\omega$, $k$ and $h(m)$ in the above are standard, thus, $\II$ truly proves``$\mbox{$X$ is $\bbomega^{k}$-large}\to\mbox{$X$ is $m$-dense$(\RT^{2}_{2})$}$''.
\end{proof}

\bibliographystyle{plain}
\bibliography{bib}

\end{document}